\documentclass{amsart}

\usepackage[english]{babel}
\usepackage[cp1250]{inputenc}
\usepackage{amssymb,amsthm,amsfonts}
\usepackage{mathrsfs}
\usepackage{verbatim}
\usepackage{url}

\newcommand{\RR}{\mathbb R}
\newcommand{\II}{\mathbb I}

\newcommand{\NN}{\mathbb N}

\newcommand{\CC}{\mathbb C}

\newcommand{\scrK}{\mathscr K}

\newcommand{\cW}{W}
\newcommand{\cF}{\mathcal F}
\newcommand{\cB}{\mathcal B}
\newcommand{\cD}{\mathcal D}
\newcommand{\cU}{\mathcal U}

\newcommand{\cI}{\mathcal I}
\newcommand{\cL}{\mathcal L}

\newcommand{\cK}{\mathcal K}

\newcommand{\cH}{\mathcal H}

\newcommand{\cA}{\mathcal A}

\newcommand{\cN}{\mathcal N}

\newcommand{\im}{\mathrm{Im}}
\newcommand{\re}{\mathrm{Re}}
\newcommand{\benu}{\begin{enumerate}}
\newcommand{\eenu}{\end{enumerate}}
\newcommand{\bop}{\begin{opomba}}
\newcommand{\eop}{\end{opomba}}

\newcommand{\Bor}{\mathrm{Bor}}

\newcommand{\B}{\mathrm{Bor}}

\newcommand{\Id}{\mathrm{Id}}
\newcommand{\id}{\mathrm{Id}}

\newcommand{\supp}{\mathrm{supp}}

\newcommand{\beqn}{\begin{align*}}
\newcommand{\eeqn}{\end{align*}}

\newcommand{\bdefi}{\begin{definition}}
\newcommand{\edefi}{\end{definition}}
\newcommand{\bcor}{\begin{corollary}}
\newcommand{\ecor}{\end{corollary}}
\newcommand{\bthe}{\begin{theorem}}
\newcommand{\ethe}{\end{theorem}}
\newcommand{\bpro}{\begin{proposition}}
\newcommand{\epro}{\end{proposition}}
\newcommand{\blem}{\begin{lemma}}
\newcommand{\elem}{\end{lemma}}
\newcommand{\brem}{\begin{remark}}
\newcommand{\erem}{\end{remark}}
\newcommand{\bequ}{\begin{equation}}
\newcommand{\eequ}{\end{equation}}
\newcommand{\bprf}{\begin{proof}}
\newcommand{\eprf}{\end{proof}}

\newtheorem{theorem}{Theorem}[section]
\newtheorem{corollary}[theorem]{Corollary}
\newtheorem{lemma}[theorem]{Lemma}
\newtheorem{proposition}[theorem]{Proposition}
\newtheorem*{problem*}{Problem}

\theoremstyle{definition}
\newtheorem{definition}[theorem]{Definition}
\newtheorem{remark}[theorem]{Remark}

\begin{document}

%-------------------------------------------------------------------------
% editorial commands: to be inserted by the editorial office
%
%\firstpage{1} \volume{228} \Copyrightyear{2004} \DOI{003-0001}
%\seriesextraline{This is the Concrete Title of this Book\br H.E. R and S.T.C. W, Eds.}
%\firstpage{1}

\title[Integral Representations of $\ast$-Rep\-res\-entations of $\ast$-algebras]{Integral Representations of $\ast$-Rep\-res\-entations of $\ast$-algebras}

%\title[Noncommutative and unbounded Gelfand-Naimark theorems]{Noncommutative and unbounded Gelfand-Naimark theorems for $\ast$-algebras}

%----------Author 1
\author{Alja\v z Zalar}

\address{%
University of Ljubljana\\
Faculty of Math.~and Phys., Dept.~of Math.\\
Jadranska 19\\
SI-1000 Ljubljana, Slovenia}

\email{aljaz.zalar@imfm.s}

%\subjclass{28B05, 46G10, 46L05, 46L10, 46L51, 47A67}


\date{\today}

\begin{abstract} 
	Regular normalized $B(\cW_1,\cW_2)$-valued non-negative spectral measures introduced in 
	\cite{Zalar2014}
	are in one-to-one correspondence with unital $\ast$-rep\-re\-sen\-tations 
	$\rho:C(X,\CC)\otimes \cW_1 \rightarrow \cW_2$, 
	where $X$ stands for a compact Hausdorff space and $\cW_1, \cW_2$ stand for von Neumann algebras. 
	In this	paper we generalize this result in two directions.
	The first is to $\ast$-representations of the form
	$\rho:\cB\otimes \cW_1\rightarrow \cW_2$, where $\cB$ stands for a commutative 
	$\ast$-algebra $\cB$, and the second is to
	special (not necessarily bounded) $\ast$-representations of the form
	$\rho:\cB\otimes \cW_1\rightarrow \cL^+(\cD)$, 
	where $\cL^+(\cD)$ stands for a $\ast$-algebra of special 
	linear operators on a dense subspace $\cD$ of a Hilbert space $\cK$.
\end{abstract}

\maketitle
\section{Introduction and Notatiton}
 
%Spectral measures and their adaptations are well-studied in the representation theory 
%(\cite{Bas-Fer-Kar}, \cite{Ber}, \cite{Con}, \cite{Mas}, \cite{Mas-Ros}, \cite{Ros}).
	A \textsl{$\ast$-representation} of a $C^{\ast}$-algebra $\cA$ is an algebra homomorphism 
$\rho:\cA \rightarrow \cW$ such that $\rho(a^\ast) = \rho(a)^{\ast}$, where $\cW$ stands for a 
von Neumann algebra.
%$B(\cH)$ stands for the $C^\ast$-algebra of bounded linear operators on a Hilbert space $\cH$.
A version of a well-known %Gelfand-Naimark 
representation theorem is the following (see \cite[p.\ 259]{Con} or \cite[5.2.6. Theorem]{Kad-Rin} and note that the $C^\ast$-algebra $B(\cH)$ of bounded linear operators on a Hilbert space $\cH$ can be replaced by a von Neumann algebra $\cW$ by \cite[Theorem 2.7.4]{Pet}).

\begin{theorem} \label{komutativne-C-star-algebre}
	Let $X$ be a compact Hausdorff space, $\cW$ a von Neumann algebra and 
	$\rho: C(X,\CC) \rightarrow \cW$ a linear map. Let $\B(X)$ be a Borel $\sigma$-algebra on $X$.
	The following statements are equivalent.
		\begin{enumerate}
			\item $\rho: C(X,\CC) \rightarrow \cW$ is a unital $\ast$-representation.
			\item There exists a unique regular normalized spectral measure 
				$E : \B(X) \rightarrow \cW$ such that $\rho(f) =	\int_X f\; dE$ for every $f \in C(X,\CC)$.
		\end{enumerate}
\end{theorem}

The motivation for our previous paper \cite{Zalar2014} was a generalization of Theorem
\ref{komutativne-C-star-algebre} to certain non-commutative $C^{\ast}$-algebras. Namely, we 
found a generalization to the $\ast$-representations of the form $\rho: C(X,\cW_1) \rightarrow \cW_2$, where $\cW_1$, $\cW_2$ stand for von Neumann algebras. For the generalization we introduced non-negative spectral measures. \textsl{Non-negative spectral measure} 
$M \colon \Bor(X) \to B(\cW_1,\cW_2)$ is a set function, such that for every hermitian projection $P\in \cW_1$ the set functions 
	$M_P:\Bor(X)\to \cW_2, M_P(\Delta):=M(\Delta)(P)$
are spectral measures and the equality 
	$M_P(\Delta_1)M_Q(\Delta_2)=M_{PQ}(\Delta_1\cap \Delta_2)$
holds for all hermitian projections $P,Q\in \cW_1$ and all sets $\Delta_1, \Delta_2\in \Bor(X)$.
The generalization is the following.

\begin{theorem} \label{reprezentacijski-izrek-2}
	Let $X$ be a compact Hausdorff space, $\B(X)$ a Borel $\sigma$-algebra on $X$, $\cW_1$, $\cW_2$ von 	
	Neumann algebras and 
		$\rho: C(X,\cW_1) \to \cW_2$
	a map.
	The following statements are equivalent.
	\benu
		\item \label{prva-tocka} $\rho: C(X,\cW_1) \to \cW_2$ is a unital
			$\ast$-representation. 
		\item	\label{druga-tocka} There exists a unique regular normalized non-negative spectral measure 
			$M:\B(X)\to B(\cW_1,\cW_2)$ such that
				$\rho(F)=\int_X F\; dM$
			for every $F\in C(X,\cW_1)$.
	\eenu
\end{theorem}

Since the common assumption in Theorems \ref{komutativne-C-star-algebre} and \ref{reprezentacijski-izrek-2} is the boundedness of $\ast$-representations, our initial motivation for research was to study unbounded versions of Theorems \ref{komutativne-C-star-algebre} and \ref{reprezentacijski-izrek-2}. Our initial setting were the algebras $C(X,\CC)$ and $C(X,\CC)\otimes W$, where $X$ is not a compact Hausdorff space. However, in the course of research, the natural question of replacing the algebra $C(X,\CC)$ by a general $\ast$-algebra $\cB$ appeared. For this purpose, Theorems \ref{komutativne-C-star-algebre} and \ref{reprezentacijski-izrek-2} have to be generalized to general $\ast$-algebras first. The motivation for this paper is the following problem.

\begin{problem*}
	Let $\cB$ be a commutative unital $\ast$-algebra and $\cW$, $\cW_1$, $\cW_2$ von Neumann algebras.
	Find one-to-one correspondence between $\ast$-representations $\rho$ of
	\begin{enumerate}
		\item[(A)] $\cB$, which map into a von Neumann algebra $\cW$, and special spectral measures.
		\item[(B)] $\cB\otimes \cW_1$, which map into a von Neumann algebra $\cW_2$, and special non-negative
			spectral measures.
		\item[(C)] $\cB$, which map into a $\ast$-algebra of all linear	operators on a densely defined 
			subspace of a Hilbert space, and special spectral measures.
		\item[(D)] $\cB\otimes \cW$, which map into a $\ast$-algebra of all linear	operators on a densely 
			defined subspace of	a Hilbert space, and special non-negative spectral measures.
	\end{enumerate}
\end{problem*}

	Before introducing terminology about $\ast$-algebra and their $\ast$-representations,
let us briefly survey some known facts about integral representations from literature. 

\subsection*{Integral representations of commutative semigroups with involution in literature}
A usual and even more general context, is the study of integral representations (via spectral measures) of $\ast$-representations of commutative semigroups $S$ with involution $\ast$ on a Hilbert space (where the operators are not necessarily bounded). The integration is over the set $S^\ast$ of all characters of $S$, i.e., functions $\chi:S\to \CC$ such that $\chi(st^\ast)=\chi(s)\chi(t)^\ast$ holds for every $s,t\in S$. $S^\ast$ is a completely regular space when equipped with the topology of pointwise convergence inherited from $\CC^S$. The cases $S=(-\infty, \infty)$ or $S=[0,\infty)$ with multiplication as the semigroup operation are well-known, see \cite[Chapter XI]{Nagy}. The extension to
general commutative semigroups with identity and involution are
\cite[Theorem 2]{Ressel-Ricker1998} and \cite[Theorem 3.1]{Atan}.  
\cite[Theorem 2]{Ressel-Ricker1998} is equivalent to Theorem \ref{komutativne-C-star-algebre} above. The main step in the proof of the equivalence is to show, that if a semigroup with an involution is induced by a unital, commutative complex algebra with an involution, where the semigroup operation is the algebra multiplication, 
then the spectral measure $E$ is supported on the subspace of linear characters. The extension to unbounded operators is also deeply studied, see \cite{Ressel-Ricker2002} and reference therein. The most general representation theorem (with unique regular spectral measure), 
where there are no topological constraints on the semigroup $S$ involved, is \cite[Theorem 1.2]{Ressel-Ricker2002}. The proof is very tehnical but the main step is the reduction to the bounded case \cite[Theorem 2]{Ressel-Ricker1998} and then the constructed spectral measure  on a certain dense subspace ($D_c$ in the notation of \cite{Ressel-Ricker2002}) is extended by continuity to the whole Hilbert space. The density of $D_c$ is included in the definition of $\ast$-representation from 
\cite{Ressel-Ricker2002} (see \cite[Definition 1.1]{Ressel-Ricker2002}). Since $D_c$ is always dense if we integrate with respect to a regular spectral measure, this is a natural requirement. 

\subsection*{Terminology on $\ast$-algebras and on $\ast$-representations} We follow the monograph \cite{Sch2}. Let $\cD$ be a dense linear subspace in a Hilbert space $\cH$.

	A \textsl{$\ast$-algebra} is a complex associative algebra $\cA$ equipped with a mapping
$a\to a^\ast$, called the \textsl{involution} of $\cA$, such that
$(\lambda a+\mu b)^\ast=\overline \lambda a^\ast+\overline \mu b^\ast$ and $(a^\ast)^\ast$ for
all $a, b \in \cA$ and all $\lambda,\mu \in \CC$.
	$\cL^{+}(\cD)$ is the set of all linear operators $a$ with domain $\cD$ for which 
$a \cD\subseteq \cD, \cD\subseteq \cD(a^{\ast})$ and 
$a^{\ast}\cD\subseteq \cD$, where $a^{\ast}$ stands for the classical adjoint of an operator $a$.
$\cL^+(\cD)$ equipped with an involution $a^+=a^\ast\upharpoonright_{\cD}$ is a $\ast$-algebra. 
Note that the identity operator $\id_{\cD}$ belongs to $\cL^+(\cD)$ and by \cite[Proposition 2.1.8]{Sch2}, all 
$a\in\cL^+(\cD)$
are closable. We write $\overline{a}$ for the \textsl{closure} of an operator $a$.

	A \textsl{$\ast$-representation} of a $\ast$-algebra $\cA$ on $\cD$ is an algebra homomorphism 
$\rho$ of $\cA$ into $\cL^+(\cD)$ such that $\rho(1)=\id_{\cD}$ and 
$\left\langle \rho(a)\phi, \psi\right\rangle = \left\langle \phi, \rho(a^\ast)\psi\right\rangle$ 
for all $\phi, \psi \in \cD$ and all $a \in \cA$.
	The \textsl{graph topology} of $\rho$ is the locally convex topology on the vector
space $\cD$ defined by the norms $h \rightarrow \|\phi\| + \|\rho(a)\phi\|$, 
where $a \in \cA$. 
If $\cD(\overline{\rho})$ denotes the completion of $\cD$ in the graph topology of 
$\rho$, then 
$\overline{\rho}(a) := \overline{\rho(a)}\upharpoonright \cD(\overline{\rho})$, $a \in \cA$, 
defines a $\ast$-representation of $\cA$ with domain $\cD(\overline{\rho})$, 
called the \textsl{closure} of $\rho$.
In particular, $\rho$ is \textsl{closed} if and only if $\cD$ is complete in 
the graph topology of $\rho$.
A closed $\ast$-representation $\rho$ of a commutative $\ast$-algebra $\cB$ 
is called \textsl{integrable} if $\overline{\rho(b^\ast)} = \rho(b)^\ast$ for all $b \in \cB$.
For several characterizations of the integrability, see \cite[Chapter 9]{Sch2}.

	\textsl{Character space} $\widehat{\cB}$ of a commutative $\ast$-algebra $\cB$ 
is the set of all non-trivial $\ast$-homo\-mor\-phisms $\chi:\cB\to\CC$. 
An element $b \in \cB$ can be viewed as a function $f_b$ on the set $\widehat{\cB}$, 
i.e., $f_b(\chi) = \chi(b)$ for $b\in \cB$ and $\chi\in \widehat{\cB}$.  
Let $\tau$ denote the weakest topology on the set $\widehat{\cB}$ for which all functions 
$f_b$ are continuous. 
This topology is generated by the sets $f^{-1}_b((c,d))$, $-\infty\leq c \leq d\leq \infty$.
The topology $\tau$ on $\widehat{\cB}$ is Hausdorff. 
	
\subsection*{Comparison of $\ast$-representations of semigroups with involution and $\ast$-algebras}
	A possible approach to solve Problems A-D is to induce a semigroup with involution from the given commutative $\ast$-algebra and use
the results on integral representations of $\ast$-representations of semigroups with involution presented above. However, there is a problem, because of the differences
in the definitions of $\ast$-representations of a $\ast$-algebra (see \cite{Sch2}) and a semigroup (see \cite{Ressel-Ricker1998} and \cite{Ressel-Ricker2002}).
The differences are the following:
	\begin{enumerate}
		\item $\ast$-representations of semigroups are not `linear' compared to $\ast$-algebras, where they are.
		\item $\ast$-representation of a semigroup maps into the set of normal operators (not necessarily having the same domain) while $\ast$-representation of a
			$\ast$-algebra maps into the set of all linear operators on a chosen dense subspace in a Hilbert space, obeying some additional conditions.
		\item $\ast$-representation of a semigroup has all the necessary properties to be an integration 
			with respect to a regular spectral measure, while in the case of a $\ast$-algebra this
			is not required. 
	\end{enumerate} 
	Hence, there are more $\ast$-representations of a $\ast$-algebra. Therefore, one has to be cautious when using the results for semigroups 
in the context of $\ast$-algebras. However, if we concentrate on $\ast$-representations with a regular representing spectral mesures, they can be used. In the next subsection we prove two properties
of such $\ast$-representations.

\subsection*{$\ast$-representations of commutative $\ast$-algebras with a regular representing spectral measures}
Let $\rho:\cB\to \cL^+(\cD)$ be a $\ast$-representation of a commutative $\ast$-algebra $\cB$ on a dense subspace $\cD$ of a Hilbert space $\cH$. We say, $\rho$ has an \textsl{integral representation}
if there is a spectral measure $E:\Bor(\widehat{\cB})\rightarrow B(\cH)$ on a Borel 
$\sigma$-algebra $\Bor(\widehat{\cB})$, such that for every $b\in \cB$ we have
	$\overline{\rho(b)}x=(\int_{\widehat{\cB}} f_b(\chi)\;dE(\chi))x$ for every 
$x\in \cD(\overline{\rho(b)})$.

The following proposition states, that every $\ast$-representation $\rho:\cB\to \cL^+(\cD)$ with an integral representation is integrable. 

\begin{proposition} \label{integrabilnost}
		Assume the notation above. If $\rho:\cB\to \cL^+(\cD)$ has an integral representation, then it is 
	integrable. 
\end{proposition}

\begin{proof}
	Since the spectral integral is always a normal operator (see (iv) of Theorem \ref{izrek2} below), 
	$\overline{\rho(b)}$ must be normal for every $b\in \cB$. By 
	\cite[Theorem 9.1.2]{Sch2}, this is true iff $\rho$ is integrable. 
\end{proof}
	
%	By Proposition \ref{integrabilnost}, a necessary condition, for a $\ast$-representation
%$\rho:\cB\to \cL^+(\cD)$ of a commutative $\ast$-algebra $\cB$ on a dense subspace $\cD$
%of a Hilbert space $\cH$, to have an integral representation, is the property of integrability. 
Now we derive a necessary condition for the regularity of the representing spectral measure 
$E:\Bor(\widehat{\cB})\rightarrow B(\cH)$
of $\rho:\cB\to \cL^+(\cD)$. We say that $E$ is \textsl{regular}, if for every $h_1,h_2\in \cH$
the complex measure
	$E_{h_1,h_2}:\Bor(\widehat{\cB})\rightarrow \CC$, defined by
		$E_{h_1,h_2}(\Delta):=\left\langle E(\Delta)h_1,h_2\right\rangle$
is regular.
Let $\scrK$ be the set of compact sets in $\widehat{\cB}$.
For a compact set $K\subseteq \Bor(\widehat{\cB})$ we define the function 
	$\alpha_K:C(\widehat\cB,\CC)\to \RR$ by
		$\alpha(K)(f):=\sup_{\chi\in K}\|f(\chi)\|,$
where $C(\widehat\cB,\CC)$ 
is the vector space of continuous functions.
	Let $\rho:\cB\to \cL^+(\cD)$ be a $\ast$-representation and $K\in \scrK$ a compact set.
Define the set
	$$D_{\alpha_K, \rho}:=\left\{ h\in \cap_{b\in \widehat \cB}\cD(\overline{\rho(b)})
			\colon \|\overline{\rho(b)}h \|
			\leq \alpha_K(b)\|h\|\right\}.$$
The following proposition states, that if a representing measure $E$ of $\rho$ is regular, then
$\cup_{K\in \scrK} D_{\alpha_K, \rho}$ is dense in $\cH$.
%	Therefore, the necessary condition for an integrable 
%	$\ast$-representation $\rho$ to have a unique regular spectral measure 
%	$E_\rho$ is the density of

\begin{proposition} \label{gostost-D-ja}
		Assume the notation above. If $\rho$ has an integral representation by a regular spectral measure 
	$E$, then the set $\cup_{K\in \scrK} D_{\alpha_K, \rho}$ is dense in $\cH$.
\end{proposition}

\begin{proof}
		%Let a $\ast$-representation $\rho$ has an integal representation by a regular spectral measure 
		%$E$. 
		First we prove, that the set $\cup_{K\in \scrK} E(K)(\cH)$ is dense in $\cH$. Pick $h\in \cH$.
	For every $K\in \scrK$ it is true that
		\begin{eqnarray*}
			\|h-E(K)h\| &=& \|E(K^c)h\|=
			(E_{h,h}(K^c))^{\frac{1}{2}}
			=	(E_{h,h}(X)-E_{h,h}(K)
					)^{\frac{1}{2}}\\
			&=& 
				(\|h\|^2-E_{h,h}(K))^{\frac{1}{2}}.
		\end{eqnarray*}
	By the regularity of $E_{h,h}$, we have 
	$\displaystyle\sup_{K\in \scrK}E_{h,h}(K)=\|h\|^2$. 	
	Hence, we conclude that 
	the set $\cup_{K\in \scrK} E(K)(\cH)$ is dense in $\cH$.
	
	Now we prove that the set $\cup_{K\in \scrK} D_{\alpha_K, \rho}$ contains the set 
	$\cup_{K\in \scrK} E(K)(\cH)$.
	%Let us define a $\ast$-representation $\widehat E:\cB\to \cN(\cH)$ by $\widehat E(b):=(\int_{\widehat B} f_b(\chi)\; dE(\chi))$.
	%Now we will show, that $\cup_{K\in \scrK} D_{\alpha_K, \widehat E}$ contains a dense set $\displaystyle\lim_{K\in \scrK}E(K)h$ and hence, it is dense.	
	For every $h\in \displaystyle\cup_{K\in \scrK} E(K)(\cH)$ there is a compact set 
	$K\in \scrK$ such that $h\in E(K)\cH$. Then for every $b\in \cB$ we have
	$$\int_{\widehat{B}}\left|f_{b}\right|^2\; dE_{h,h}=
		\int_{K}\left|f_{b}\right|^2\; dE_{h,h}\leq
		\alpha_K(b)^2 \int_{K}1\; dE_{h,h}\leq \alpha_K(b)^2 \|h\|^2.
		$$
	Hence, $\cup_{K\in \scrK} E(K)(\cH)\subseteq \cup_{K\in \scrK} D_{\alpha_K, \rho}$. 
	Thus, $\cup_{K\in \scrK} D_{\alpha_K, \rho}$ is dense in $\cH$.
\end{proof}	
	
By Propositions \ref{integrabilnost} and \ref{gostost-D-ja}, $\ast$-representations 
	$\rho:\cB\rightarrow \cL^+(\cD)$ with a representing spectral measure 
	$E:\Bor(\widehat{\cB})\rightarrow B(\cH)$ 
are integrable and have a dense set $\cup_{K\in \scrK} D_{\alpha_K, \rho}$.

\subsection*{Problem C - known result}

	For a special case of a countably generated commutative unital $\ast$-algebra $\cB$, the solution to
Problem C is the following result (see \cite[Theorem 7]{Sav-Sch}).

\renewcommand{\thetheorem}{C}\begin{theorem}\label{Sav-Schjev-rezultat}
	Suppose that $\cB$ is a countably generated commutative unital $\ast$-algebra. Suppose
	$\widehat{\cB}$ is equipped with the Borel structure induced by the weakest topology 
	for which all functions $f_b$, $b \in \cB$, are continuous.
	Let $\rho:\cB\rightarrow \cL^+(\cD)$ be a closed $\ast$-representation of $\cB$ on a dense
	linear subspace $\cD$ of a Hilbert space $\cH$. 
	The following statements are equivalent.
		\begin{enumerate}
			\item $\rho$ is an integrable $\ast$-representation of $\cB$.
			\item There exists a unique spectral measure $E_\rho:\Bor(\widehat{\cB})\to B(\cH)$
				such that 
					$\overline{\rho(b)} =\int_{\widehat{\cB}}f_b(\chi)\;dE(\chi)$ for all $b\in \cB$.
		\end{enumerate}
\end{theorem}

\addtocounter{theorem}{-1}
\renewcommand{\thetheorem}{\arabic{section}.\arabic{theorem}}

%	By Proposition \ref{integrabilnost}, the condition (1) is necessary for a $\ast$-representation to have representing spectral measure. 
Theorem \ref{Sav-Schjev-rezultat} gives a one-to-one correspondence between integrable $\ast$-rep\-re\-sen\-tations of a countably generated commutative unital $\ast$-algebra $\cB$ on a dense linear subspace $\cD$ of a Hilbert space $\cH$ and spectral measures $E$ on $\Bor(\widehat{\cB})$ such that $(\int_{\widehat{\cB}}f_b(\chi)\;dE(\chi))\cD\subseteq \cD$ for every $b\in \cB$. 
%Let us first explain that the last condition on the spectral measure $E$ is necessary.
%Since a $\ast$-representation 
%$\rho$ maps into $\cL^+(\cD)$, we must have 
%$\rho(b)^\ast \cD \subseteq \cD$ for every $b\in \cB$. 
%If $\overline{\rho(b)}=\int_{\widehat{\cB}}f_b(\chi)\;dE(\chi)$ for some spectral measure $E$ on
%$\widehat{\cB}$, then $\rho(b^\ast)^\ast=\int_{\widehat{\cB}}f_b(\chi)\;dE(\chi)$. Therefore $(\int_{\widehat{\cB}}f_b(\chi)\;dE(\chi))\cD\subseteq \cD$ for every $b\in \cB$. 
%Now we explain that the correspondence is indeed one-to-one. Suppose $E_1$ and $E_2$ are two spectral measures, satisfying $\widehat{E_j}(b):=(\int_{\widehat{\cB}}f_b(\chi)\;dE_j(\chi))\cD\subseteq \cD$ for every $b\in \cB$ and $j=1,2$, such that $\widehat{E_1}\upharpoonright_{\cD}=\widehat{E_2}\upharpoonright_{\cD}: \cB\rightarrow \cL^+(\cD)$. 
%A $\ast$-representation $\widehat{E_1}\upharpoonright_{\cD}=\widehat{E_2}\upharpoonright_{\cD}$ 
%of $\cB$ on $\cD$ is integrable and has representing spectral measures $E_1$ and $E_2$.
%By the uniqueness of the representing spectral measure in Theorem \ref{Sav-Schjev-rezultat}, we
%conclude that $E_1=E_2$.
However, Theorem \ref{Sav-Schjev-rezultat} covers only countably generated commutative unital $\ast$-algebras $\cB$, e.g., $\cB:=\CC[x_1,\ldots,x_n]$. The case $\cB=C(X,\CC)$, where $X$ is a topological space, is not covered. 

%For a satisfactory solution to Problem 
%C we would like to know something about this case as well. In Theorem 
%\ref{komutativne-C-star-algebre} the unique representing measures are regular. Therefore, we will 
%concentrate on regular measures also in the unbounded case. Problem C becomes the
%following.
%
%\renewcommand{\theproblem}{C'}\begin{problem}
	%Suppose that $\cB$ is a commutative unital $\ast$-algebra and $\cD$ a dense linear subspace 
	%of a Hilbert space $\cH$. What conditions must a closed $\ast$-representation 
	%$\rho:\cB\rightarrow \cL^+(\cD)$ satisfy to have a unique regular representing spectral measure 
	%$E_\rho:\Bor(\widehat{\cB})\to B(\cH)$, i.e.\ 
	%$\overline{\rho(b)} =\int_{\widehat{\cB}}f_b(\chi)\;dE(\chi)$ for all $b\in \cB$?
%\end{problem}
%
%\addtocounter{problem}{-1}
%\renewcommand{\theproblem}{\arabic{section}.\arabic{problem}}

\subsection*{Problem - new results}

	In what follows $\cB$ will be a commutative unital $\ast$-algebra, $\widehat{B}$ the character space of $\cB$ equipped with the Borel structure $\Bor(\widehat{B})$ induced by the weakest topology 
	for which all functions $f_b$, $b \in \cB$, are continuous, $\cH$, $\cK$ Hilbert spaces and $\cW\subseteq B(\cH), \cW_1\subseteq B(\cH)$ von Neumann algebras. 
	
	Recall, that the \textsl{support} of a spectral measure 
$E : \B(\widehat{\cB}) \rightarrow \cW$, denoted by $\supp(E)$, is the set
$\displaystyle\overline{\cup_{h\in \cH}\;\supp(E_{h,h})}$ in $\widehat{B}$.
	The following is the solution to Problem A.

\renewcommand{\thetheorem}{A}\begin{theorem} \label{theorem-A} 
		%Let $\cB$ be a commutative $\ast$-algebra and $B(\cH)$ a $C^{\ast}$-algebra of all bounded linear 
	%operators on a Hilbert space $\cH$. Let $\widehat{\cB}$ be the character space of $\cB$ and
	%$\B(\widehat{\cB})$ a Borel $\sigma$-algebra on $\widehat{\cB}$.
	The following statements are equivalent.
		\begin{enumerate}
			\item $\rho: \cB \rightarrow \cW$ is a unital $\ast$-representation.
			\item There exists a unique regular spectral measure 
				$E : \B(\widehat{\cB}) \rightarrow \cW$ with a compact support such that 
					$\rho(b) =	\int_{\widehat{\cB}} f_b(\chi)\; dE(\chi)$
				for all $b \in \cB$.
		\end{enumerate}
\end{theorem}

\addtocounter{theorem}{-1}
\renewcommand{\thetheorem}{\arabic{section}.\arabic{theorem}}

	For every $F:=\sum_{j=1}^m b_j\otimes A_j\in \cB\otimes \cW_1$ we define a map
		$f_F(\chi):\widehat{B}\otimes \cW_1\to \cW_1$ by
			$f_F(\chi):=\sum_{j=1}^m f_{b_j}(\chi)\otimes A_j.$
The \textsl{support} of a non-negative spectral measure 
$M:\B(\widehat{\cB}) \rightarrow B(\cW_1,\cW_2)$, denoted by  $\supp(M)$, 
is the support of the spectral measure 
	$M_{\id_{\cH}}: \B(\widehat{\cB}) \rightarrow \cW_2,
		M_{\id_{\cH}}(\Delta):=M(\Delta)(\id_{\cH}).$
Using Theorem A we obtain the following solution to Problem B.
	
\renewcommand{\thetheorem}{B}\begin{theorem} \label{theorem-B} 
	%	Let $\cB$ be a commutative $\ast$-algebra, $B(\cH)$ and $B(\cK)$ $C^{\ast}$-algebra of all bounded 
	%linear operators on a Hilbert space $\cH$ and $\cK$. 
	%Let $\widehat{\cB}$ be the character space of $\cB$ and	$\B(\widehat{\cB})$ a Borel 
	%$\sigma$-algebra on $\widehat{\cB}$.
	The following statements are equivalent.
		\begin{enumerate}
			\item $\rho: \cB \otimes \cW_1\rightarrow \cW_2$ is a unital $\ast$-representation.
			\item There exists a unique regular normalized non-negative spectral measure 
				$M : \B(\widehat{\cB}) \rightarrow B(\cW_1,\cW_2)$ with a compact support such that 
					$\rho(F) =	\int_{\widehat{\cB}} f_F(\chi)\; dM(\chi)$
				for every $F \in \cB\otimes \cW_1$.
		\end{enumerate}
\end{theorem}

\addtocounter{theorem}{-1}
\renewcommand{\thetheorem}{\arabic{section}.\arabic{theorem}}
	
	%$$we denote the function $\sum_{j=1}^m f_{b_j}\otimes A_j$ on $\widehat B$, 
%defined by 	
	%$(\sum_{j=1}^m f_{b_j}\otimes A_j)(\chi)=\sum_{j=1}^m f_{b_j}(\chi)\otimes A_j.$
%Let $S_\ell(A_j)=\sum_{k=1}^{n_{\ell}} \zeta_{k,\ell;j} P_{k,\ell;j}$ be the sequence from $\cW$ converging to $A_j$ in the norm topology, where for each $j=1,\ldots,m$, $\zeta_{k,\ell;j}\geq 0$ and 
%$P_{k,\ell;j}$ are pairwise orthogonal hermitian projections. Such a sequence exists by Lemma
%\ref{lema10} below. 
%Then $(\int_{\widehat{\cB}} f_F(\chi)\; dM(\chi)) x$ is defined by
	%$$(\int_{\widehat{\cB}} f_F(\chi)\; dM(\chi)) x
		%:=	\lim_{\ell\to \infty}
			%\sum_{j=1}^m \sum_{k}\zeta_{k,\ell;j}
			%(\int_{\widehat{\cB}} f_{b_j}(\chi)\; dM_{P_{k,\ell;j}}(\chi)),$$
%where $M_{P_{k,\ell;j}}:\Bor(\widehat{B}) \to B(\cH)$ are spectral measures.
%The well-definedness of this definition was proven in \cite{Zalar2014} for operators $f_{F}$, such
%that every function $f_{b_j}$ is $M_{P}$-essentially bounded for every hermitian projection $P\in \cW$,
%i.e.\ there exists a constant $K_{f_{b_j}}\in \RR$, such that 
%$\int_{\widehat{B}} \left|f_{b_j}(\chi)\right|\; d\left\langle M_P(\chi)h,h\right\rangle
%\leq K_{f_{b_j}}\|h\|^2$ for every $h\in \cH$. But for such $f_F$, the integral
%$\int_{\widehat{\cB}} f_F(\chi)\; dM(\chi)$ is a bounded operator on $\cH$. The main purpose of
%this paper is to extend this construction to a bigger set of operators $f_F$, such that the functions  
%$f_{b_j}$ are not necessarily $M_P$-essentially bounded for every hermitian projection $P\in \cW$. This is done in Section 3.

	Let us now study Problem C. If we concentrate only on regular spectral measures, then 
Propositions \ref{integrabilnost} in \ref{gostost-D-ja} give necessary conditions
for $\rho$ to have a representing measure. Theorem C' states, that they are also sufficient.
	
\renewcommand{\thetheorem}{C'}
\begin{theorem}\label{resitev-problem-1'}\label{theorem-C'}
	Let $\cD$ be a dense linear subspace $\cD$ of a Hilbert space $\cH$  
	and $\rho:\cB\rightarrow \cL^+(\cD)$ a $\ast$-representation.
	The following statements are equivalent.
		\begin{enumerate}
			\item $\rho$ is integrable and the set 
				$\displaystyle\cup_{K\in \scrK} D_{\alpha_K, \rho}$ is dense in $\cH$.
			\item There exists a unique regular spectral measure 
				$E:\Bor(\widehat{\cB})\to B(\cH)$, such that 
				$\supp(E_{h,h})$ is compact for every $h\in \cup_{K\in \scrK} D_{\alpha_K, \rho}$
				and the equality
					$\overline{\rho(b)}x =\int_{\widehat{\cB}}f_b(\chi)\;dE(\chi)\;x$ holds for every 
				$x\in \cD(\overline{\rho(b)})$ and all $b\in \cB$.
		\end{enumerate}
\end{theorem}
	
\addtocounter{theorem}{-1}
\renewcommand{\thetheorem}{\arabic{section}.\arabic{theorem}}

		Let $\cD$ be a dense linear subspace of a Hilbert space $\cK$
	and $\rho:\cB\otimes \cW\rightarrow \cL^+(\cD)$ a $\ast$-representation. For every hermitian projection $P\in B(\cH)$ the map 
	$\rho_P:\cB \to \cL^+(\cD),$ defined by $\rho_P(b):=\rho(b\otimes P)$,
is a $\ast$-representation.	Using Theorem C' we obtain the following solution to Problem D.
	%Define a set 
	%$D_{\alpha_K, \overline{\rho_{\id_{\cH}}}}:=
	%	\left\{ h\in \cap_{b\in \widehat \cB}\cD(\overline{\rho(b)})
	%		\colon \|\overline{\rho_{\id_{\cH}}(b)}h \|
	%		\leq \alpha_K(b)\|h\|\right\}.$

\renewcommand{\thetheorem}{D}\begin{theorem}	\label{neomejen-primer-posplosen-uvod-splosna-verzija}
\label{theorem-D}
		For a $\ast$-representation $\rho:\cB\otimes \cW\rightarrow \cL^+(\cD)$ 
	the following statements are equivalent.
		\begin{enumerate}
				\item The map $\rho_P$ is integrable $\ast$-representation for every hermitian projection $P\in 
					\cW$ and the set $\cup_{K\in \scrK}\cD_{\alpha_K, \rho_{\id_{\cH}}}$ is dense in $\cK$.
				\item There exists a unique regular normalized non-negative spectral measure 
					$M:\B(\widehat{\cB})\to B(\cW,B(\cK))$
					such that 
						$\overline{\rho(F)}x= \int_{\widehat{\cB}} f_F(\chi)\; dM(\chi)\; x$
					holds for every $x\in \cD(\overline{\rho(F)})$ and all $F\in \cB\otimes \cW$.
			\end{enumerate}
	\end{theorem}
	
\addtocounter{theorem}{-1}
\renewcommand{\thetheorem}{\arabic{section}.\arabic{theorem}}

	The paper is organized in the following way. In Section \ref{section3} we first introduce some theory needed throughout the paper, i.e.,
	in Subsection \ref{subsection3} we recall a theory of non-negative spectral measures and integration with respect to them, in Subsection \ref{subsection1} we repeat the integration of unbounded functions with respect to the spectral measure and finally, and in Subsection \ref{subsection2}, we present
two results on the integral representation of $\ast$-representations of a commutative semigroup with an involution. Theorems A and B are proved in Section \ref{solutions-to-A-B}, while Theorem C' in Section 
\ref{sekcija-problema-C}. Before proving Theorem D, the theory of integration 
with respect to the non-negative spectral measure has to be extended to the integration of unbounded functions. This is done in Section \ref{razsiritev-integracije}. Finally, in Section 
\ref{resitev-problema-D}, Theorem D is proved.

\section{Preliminaries} \label{section3} 

In Subsection \ref{subsection3} we present the non-negative spectral 
measures, which we introduced in \cite{Zalar2014} to prove a theorem on the integral representation of 
the $\ast$-representation of a $C^\ast$-algebra (see Theorem \ref{reprezentacijski-izrek-2}). 
In Subsection \ref{subsection1} we present a spectral theory of unbounded functions on a Hilbert space
(see \cite{Bir-Sol} or \cite{Sch2012}). In Subsection \ref{subsection2} we recall results for integral representations of $\ast$-representations of commutative semigroups with an involution from \cite{Ressel-Ricker1998}, \cite{Ressel-Ricker2002}.

\subsection{Non-negative spectral measures}
\label{subsection3}

	Let $(X,\Bor(X),\cW_1,\cW_2)$ be a measure space, i.e., $X$ is a topological space, $\Bor(X)$ a $\sigma$-algebra on $X$,
$\cW_1\subseteq B(\cH)$, $\cW_2\subseteq B(\cK)$ von Neumann algebras, 
where $B(\cH), B(\cK)$ denote
the bounded linear operators on Hilbert spaces $\cH$, $\cK$.
We denote by $\cW_p$, $\cW_+$ the subsets of all hermitian projections and all positive operators of a 
von Neumann algebra $\cW$. By a hermitian projection we mean an operator $P$, which satisfies $P=P^\ast=P^2$
and by a positive operator we mean a hermitian operator $A$, such that 
$\left\langle Ah,h\right\rangle\geq 0$ for every $h\in \cH$, where $\cH$ is a Hilbert space with an
inner product $\left\langle \cdot,\cdot\right\rangle$, such that
$\cW\subseteq B(\cH).$
\textsl{Non-negative spectral measure} $M \colon \Bor(X) \to B(\cW_1,\cW_2)$ is a set function, if for every hermitian projection $P\in (\cW_1)_p$ the set functions 
	$M_P:\Bor(X)\to \cW_2$, defined by $M_P(\Delta):=M(\Delta)(P)$,
are spectral measures and the equality 
		$M_P(\Delta_1)M_Q(\Delta_2)=M_{PQ}(\Delta_1\cap \Delta_2)$
holds for all hermitian projections $P,Q\in (\cW_1)_p$ and all sets $\Delta_1, \Delta_2\in \Bor(X)$.
Before writing a characterization of non-negative spectral measures, we need the following definition.
		
	\begin{definition}
		Let $W$ be a von Neumann algebra and $A\in W$ a hermitian operator. Let $E:\B([a,b])\to W$ be
	its representing spectral measure, i.e.,\
	$A=\int_{[a,b]} \lambda\;dE(\lambda)$ with the spectrum $\sigma(A)$ of $A$ lying in $[a,b]$.
	For $\ell\in \NN$, let $\mathcal Z_\ell=\{\lambda_0,\lambda_1,\ldots,\lambda_{n_\ell}\}$ be a partition of $[a-1,b]$, such that 
		$a-1<\lambda_{0,\ell}<a<\lambda_{1,\ell}<\ldots<\lambda_{1,n_\ell}<b$ and $\mathcal Z_\ell\subset \mathcal Z_{\ell+1}$.
	Define a sequence $S_\ell(A)$ of Riemann sums of the form
		%$S_\ell(A)=\sum_{k=1}^{n_{\ell}} \zeta_{k,\ell} (E(\lambda_{k,\ell})-E(\lambda_{k-1,\ell}))=:
		%		\sum_{k=1}^{n_{\ell}} \zeta_{k,\ell} R_{k,\ell},$
		\begin{eqnarray*} \label{Riemannove-vsote-operatorja} 
			S_\ell(A)&=&\sum_{k=1}^{n_{\ell}} \zeta_{k,\ell} (E(\lambda_{k,\ell})-E(\lambda_{k-1,\ell}))=:
				\sum_{k=1}^{n_{\ell}} \zeta_{k,\ell} R_{k,\ell},\\
		\end{eqnarray*}
	where the family $\left\{E(\lambda)\mid \lambda\in\RR\right\}$ is the resolution of identity corresponding to the spectral measure $E$
	and $\zeta_{k,\ell}\in [\lambda_{k-1,\ell},\lambda_{k,\ell}]$. 
	Let $\left|\mathcal Z_\ell\right|:=\max_{k}\left|\lambda_{k,\ell}-\lambda_{k-1,\ell}\right|$. We can choose $\left|\mathcal Z_\ell\right|$ small enough, so that
	$\|A-S_\ell(A)\|\leq \frac{1}{\ell}.$
		A sequence $S_\ell(A)$ is called a \textsl{limiting sequence} of $A$.
	\end{definition}
	
A characterization of non-negative spectral measures \cite[Theorem 8.1]{Zalar2014} is the following.

\begin{theorem}\label{karakterizacija-nenegativnih-spektralnih-mer}
	Let $(X,\Bor(X),\cW_1,\cW_2)$ be a measure space and $\left\{E_P\right\}_{P\in (\cW_1)_p}$ a family of 
	spectral measures $E_P:\Bor(X)\to \cW_2$.
	
	There is a unique non-negative spectral measure $M$ such that
		$$M_P=E_P$$
	for all hermitian projections $P\in (\cW_1)_p$ iff the following conditions hold.
		\begin{equation}\label{Pogoj1111}  
			\sum_{i=1}^{n}\lambda_i E_{P_i}(\Delta)=\sum_{j=1}^m \mu_j E_{Q_j}(\Delta)
		\end{equation}		
	for all hermitian projections $P_i,Q_j\in (\cW_1)_p$, all real numbers $\lambda_i,\mu_j \in\RR$, and all sets $\Delta\in\Bor(X)$ such that
		$\sum_{i=1}^{n}\lambda_i P_i=\sum_{j=1}^m \mu_j Q_j$,
	for each set $\Delta\in\Bor(X)$ there exists a constant $k_\Delta\in\RR^{>0}$ such that  
		\begin{equation}\label{Pogoj2222}  
			\|E_P(\Delta)\|\leq k_\Delta
		\end{equation}
	for all hermitian projections $P\in (\cW_1)_p$, and 
	for all hermitian projections $P,Q\in (\cW_1)_p$ and all sets $\Delta_1,\Delta_2\in \Bor(X)$ the 	
	equality
		\begin{equation}\label{Pogoj3333}  
			E_P(\Delta_1)E_Q(\Delta_2)=\lim_{\ell\to\infty}
			\sum_{j=0}^3 i^j \sum_{k=1}^{n_\ell}\zeta_{k,\ell,j} E_{R_{k,\ell,j}}(\Delta_1\cap \Delta_2) 
		\end{equation}	
	holds for limiting sequences $S_{\ell,j}:=\sum_{k=1}^{n_\ell} \zeta_{k,\ell,j} R_{k,\ell,j}$ 
	of the operators
		$\re(PQ)_+$, $\im(PQ)_+$, $\re(PQ)_-$, $\im(PQ)_-$
	for $j=0, 1, 2, 3$ respectively,
	where the decomposition of $PQ$ into the linear combination of positive parts is
		$PQ= \re(PQ)_+-\re(PQ)_-+i\cdot\im(PQ)_+-i\cdot\im(PQ)_-.$
\end{theorem}

From the proof of \cite[Theorem 8.1]{Zalar2014}, we extract the following proposition, which will be needed in this paper.

	\begin{proposition} \label{trditev-o-generiranju-pozitivne-mere}
				Suppose $\left\{E_P\right\}_{P\in (\cW_1)_p}$ is a family of spectral measures $E_P:\Bor(X)\to \cW_2$ satisfying conditions (1) and (2) of Theorem 
			\ref{karakterizacija-nenegativnih-spektralnih-mer}. For any positive operator $A\in \cW_1$ with a limiting sequence $S_{\ell}(A)$, 
			the measure 
				$$E_A:\Bor(X) \to \cW_2,\quad E_A(\Delta):=\lim_{\ell\to\infty}E_{S_{\ell}(A)}(\Delta),$$
			is a positive operator-valued measure.
			Moreover, the definition of $E_A$ is independent of the choice of a limiting sequence $S_{\ell}(A)$.
	\end{proposition}
 
Let $X$ be a topological space and $\B(X)$ a Borel $\sigma$-algebra on $X$. 
Non-negative spectral measure $M$ is \textsl{regular} if the spectral measures $M_P$ are regular for every hermitian projection $P\in (\cW_1)_p$, i.e., complex measures
	$(M_P)_{k_1,k_2}:\B(X)\to \CC$, $(M_P)_{k_1,k_2}(\Delta):=\left\langle M_P(\Delta)k_1,k_2	
	\right\rangle$
are regular for every $k_1, k_2\in \cK$ and every $P\in (\cW_1)_p$. 
%(see Section \ref{preliminaries}).
$M$ is \textsl{normalized} if $M(X)(\id_{\cH})=\id_{\cK}$, where $\id_\cH$, $\id_\cK$ denote the identity operators on $\cH$, $\cK$ respectively.

	A $\Bor(X)$-measurable complex function $f:X\to\CC$ is integrable with respect to a spectral measure 
$M_P:\Bor(X)\to \cW_2$ with $P\in (\cW_1)_p$, if there exists a constant $K_f\in \RR$, such that for every $k\in \cK$ we have $\int_X \left|f\right|\;d(M_{P})_{k,k}\leq K_f\|k\|^2$. A function $f$ is \textsl{$M$-integrable}, if it is $M_P$-integrable for every $P \in (W_1)_p$. Then the integral of $f$ with respect to a positive operator-valued measure $M_A$ for $A\in (\cW_1)_+$ is defined by $\int_{X}f\; dM_A:=\int_{X}f\; dM_{S_{\ell}(A)}$ for a
limiting sequence $S_{\ell}(A)$ of $A$ (it is independent from the choice of the limiting sequence). Finally, the integral $\int_{X}f\; dM_A$ of $f$
with respect to a signed operator-valued measure $M_A$ for $A\in \cW_1$ is defined by
	$\int_{X}f\; dM_{\re(A)_+}-\int_{X}f\; dM_{\re(A)_-}+
	i\cdot\int_{X}f\; dM_{\im(A)_+}-i\cdot\int_{X}f\; dM_{\im(A)_-},$
where $\re(A)$, $\im(A)$ are the real and the imaginary part of $A$ and $A_+$, $A_-$ are the positive and the negative part of $A$. The set of all $M$-integrable functions is denoted by $\cI(M)$. The map 
$\cB:\cI(M)\times \cW_1\to \cW_2$, defined by $\cB(f,A):=\int_Xf\; dM_A$, is bilinear. $\cB$ extends to a linear map $\overline{\cB}:\cI(M)\otimes \cW_1\to \cW_2$ and we define an \textsl{integral} of 
$F\in \cI(M)\otimes \cW_1$ with respect to $M$ by $\int_X F\; dM:= \overline{\cB}(F).$ The algebraic properties of an integral with respect to $M$ are given in the following proposition (see \cite[Proposition 3.5]{Zalar2014} and \cite[Proposition 7.2]{Zalar2014}).

\begin{proposition}
	For $F, G\in \cI(M)\otimes \cW_1$, $A\in \cW_1$ and $\lambda\in \CC$ we have:
	\begin{enumerate}	
		\item $\int_X (F+G)\; dM=\int_X F\; dM+\int_X G\; dM$.
		\item $\int_X \lambda F\; dM=\lambda\int_X F\; dM$.
		\item For every $\Delta\in \Bor(X)$, $\int_X \chi_{\Delta}\otimes A\; dM=M_A(\Delta)$.
		\item From $F\succeq 0$, it follows that $\int_X F\; dM\succeq 0$.
		\item $\int_X FG\; dM=(\int_X F\; dM)(\int_X G\; dM)$.
	\end{enumerate}
\end{proposition}

%Non-negative spectral measures were introduced in \cite{Zalar2014} to prove 
%Theorem \ref{reprezentacijski-izrek-2} above.

\subsection{Spectral integral on a Hilbert space} \label{subsection1}

We recall the integration of unbounded functions with respect to the spectral measure (see \cite{Sch2012}).
Let $X$ be a set, $\Bor(X)$ a Borel $\sigma$-algebra on $X$, $\cH$ is a Hilbert space and $E:\Bor(X)\to B(\cH)$ a spectral measure. Let $\cU$ be the set of all $\Bor(X)$-measurable functions $f:X\to \CC$. A sequence $(\Delta_n)_{n\in\NN}$ of sets $\Delta_n\in \Bor(X)$ is a \textsl{bounding sequence} for a subset $\cF$ of $\cU$ if each function $f\in \cF$ is bounded on $\Delta_n$, $\Delta_n\subseteq \Delta_{n+1}$ for $n\in\NN$, and 
$E(\cup_{n=1}^{\infty}\Delta_n)=\id_{\cH}.$

If $(\Delta_n)_{n\in\NN}$ is any bounding sequence, then by the properties of the spectral measure,
$E(\Delta_n)\preceq E(\Delta_{n+1})$ for $n\in\NN$ and $\lim_{n\to\infty}E(\Delta_n)x=x$ for $x\in\cH.$
Each finite set of functions $f_1,f_2,\ldots,f_r\in \cU$ has a bounding sequence 
	$\Delta_n:=\left\{t\in X\colon \left|f_j(t)\right|\leq n\;
		\mathrm{for}\;j=1,2,\ldots,r\right\}.$

\textsl{The spectral integral} $\II(f):=\int_X f\;dE$ of a function $f\in\cU$ is given by 
the following theorem (see \cite[Theorem 4.13]{Sch2012}).

\begin{theorem} \label{izrek3}
	Suppose that $f\in \cU$ and define 
		$$\cD(\mathbb{I}(f)):=\{x\in \cH\colon \int_X \left|f(t)\right|^2\;
		d\left\langle E(t)x,x\right\rangle<\infty\}.$$
	Let $(\Delta_n)_{n\in\NN}$ be a bounding sequence for $f$. Then we have:
	\begin{enumerate}
		\item 
			A vector $x\in \cH$ is in $\cD(\mathbb{I}(f))$ iff the sequence $(\II(f\chi_{\Delta_{n}})x)_{n
			\in\NN}$ converges in $\cH$, or equivalently, if $\sup_{n\in\NN}\|\II(f\chi_{\Delta_{n}})x
			\|<\infty$.
		\item For $x\in \cD(\II(f)),$ the limit of the sequence $(\II(f\chi_{\Delta_{n}})x)$ does not depend
			on the bounding sequence $(\Delta_n)$. There is a linear operator $\II(f)$ on $\cD(\II(f))$ defined
			by $\II(f)x=\displaystyle\lim_{n\to \infty}\II(f\chi_{\Delta_n})x$ for $x\in\cD(\II(f))$.
				%\begin{eqnarray*}
				%	&&\II(f)x=\lim_{n\to \infty}\II(f\chi_{\Delta_n})x\quad \mathrm{for}\;x\in\cD(\II(f)).
				%\end{eqnarray*}
		\item The vector space $\cup_{n=1}^{\infty}E(\Delta_n)\cK$ is contained in $\cD(\II(f))$ and is a core for $\II(f)$.
			Furthermore,
				$E(\Delta_n)\II(f)\subseteq \II(f)E(\Delta_n)=\II(f\chi_{\Delta_n}).$
	\end{enumerate}
\end{theorem}
$\cD(\mathbb{I}(f))$ from Theorem \ref{izrek3} is the \textsl{domain} of the operator $\II(f)$.
By $\overline{A}$ we denote the closure of a closable linear operator $A$, defined on a dense subspace of a Hilbert space $\cH$.
The main algebraic properties of the map $f\to\II(f)$ are given in the following theorem (see
\cite[Theorem 4.16]{Sch2012}).

\begin{theorem} \label{izrek2}
	For $f,g\in \cU$ and $\alpha,\beta\in\CC$ we have:
		\begin{enumerate}
			\item $\II(\overline{f})=\II(f)^\ast$,
			\item $\II(\alpha f+ \beta g)=\overline{\alpha \II(f)+\beta \II(g)}$,
			\item $\II(fg)=\overline{\II(f)\II(g)},$
			\item $\II(f)$ is a normal operator on $\cK$, and 
				$\II(f)^\ast\II(f)=\II(\overline{f}f)=\II(f)\II(f)^\ast,$
			\item $\cD(\II(f)\II(g))=\cD(\II(g))\cap \cD(\II(fg))$.
		\end{enumerate}
\end{theorem}

%\begin{remark}\label{remark2}
	To emphasize with respect to which spectral measure we integrate, we will denote the integral
	$\II(f)$ with respect to $E$ by $\II_E(f)$.
%\end{remark}

\subsection{$\ast$-representations of commutative semigroups with an involution}
\label{subsection2}

	In this subsection we present integral representation theorems \cite[Theorem 2]{Ressel-Ricker1998} and \cite[Theorem 1.2.]{Ressel-Ricker2002} for $\ast$-representations of a commutative semigroup $S$ with identity element $e$ and an involution $\ast$ (i.e., $(s^\ast)^\ast=s$ and $(st)^\ast=s^\ast t^\ast$ for all $s,t\in S$). A function $\chi:S\to\CC$ which satisfies $\chi(e)=1$ and $\chi(st^\ast)=\chi(s)\overline{\chi(t)}$ for all $s,t\in S$ is called a \textsl{character} of $S$. By 
$S^\ast$ we denote the set of all characters of $S$.  The set $S^\ast$ is a completely regular space when equipped with the topology of pointwise convergence from $\CC^S$. Let $\cH$ be a Hilbert space
and $\cW\subseteq B(\cH)$. The map $\rho:S\to \cW$ is a $\ast$-represesntation of $S$ if $\rho(e)=\id_\cH$ and 
$\rho(st^\ast)=\rho(s)\rho(t)^\ast$ for all $s,t\in S$. For $s\in S$, the function
$f_s:S^\ast\to \CC$ is defined by $f_s(\chi)=\chi(s).$ \cite[Theorem 2]{Ressel-Ricker1998} is
the following result.

\begin{theorem} \label{polgrupe-omejen}
	Let $\rho:S\to \cW$ be a $\ast$-representation of a semigroup $S$. Then there exists a unique regular spectral measure $E:\Bor(S^\ast)\to \cW$ with a compact support, such that
		$\rho(s)=\int_{S^\ast} f_s(\chi)\;dE(\chi)$
	for all $s\in S$.
\end{theorem}

\begin{remark}
	Originally, in \cite[Theorem 2]{Ressel-Ricker1998}, a von Neumann algebra $\cW$ is $B(\cH)$. Now we 	
	explain, why we can replace $B(\cH)$ by $\cW$. By \cite[Theorem 2]{Ressel-Ricker1998}, 
	$\rho:S\to \cW\subseteq B(\cH)$ is represented by	a unique regular spectral measure 
	$E:\Bor(S^\ast)\to B(\cH)$ with a compact support $K$.
	Define the sets $\cA_K=\left\{(f_b)|_K\in C(K)\colon b\in \cB\right\}$ and
	$B_{\cA_K}:=\left\{f\in \cA_K\colon \|f\|_{\infty}\leq 1\right\}$.
	By \cite[p.\ 95]{B-M-K}, $B_{\cA_K}$ is dense in the ball of $C(K)$ equipped with a supremum norm.
	Therefore, $\rho$ can be extended by the continuity to the $\ast$-representation
	$\tilde\rho:C(K)\to \cW$. By Theorem \ref{komutativne-C-star-algebre}, $\tilde\rho$
	is represented by	a unique regular spectral measure 
	$\tilde E:\Bor(K)\to \cW$. From the uniqueness, $E=\tilde E$ and thus, 
	$E:\Bor(S^\ast)\to \cW$ maps into $\cW$ as desired.	
\end{remark}

Now we explain the extension of Theorem \ref{polgrupe-omejen} to a $\ast$-representation with a range in (not necessarily bounded) normal operators. 
A function $\alpha:S\to [0,\infty)$ is an \textsl{absolute value} if 
	$\alpha(s^\ast)=\alpha(s)$, $\alpha(e)=1$, and $\alpha(st)\leq \alpha(s)
	\alpha(t)$ for all $s,t\in S$. The family of all absolute values is denoted by $\cA(S)$.
	Given a map $\rho:S\to \cN(\cH)$, where $\cN(\cH)$ is a vector space of all normal (not necessarily bounded) operators on a Hilbert space $\cH$ and $\alpha\in \cA(S)$, define the set
	$D_{\alpha}:=\{h\in \bigcap_{s\in S} \cD(\rho(s))\colon
		\|\rho(s)h\|\leq \alpha(s)\|h\|\;
		\mathrm{ for}\;\mathrm{ all }\; s\in S\},$
where $\cD(\rho(s))$ denotes the domain of $\rho(s)$.
By \cite[Definition 1.1.]{Ressel-Ricker2002}, a map $\rho:S\to \cN(\cH)$ is called a \textsl{$\ast$-representation}, if:
	\benu
		\item $\rho(e)=\id_\cH$, where $\id_\cH$ is the identity operator on $\cH$.
		\item $\rho(s^{\ast})=\rho(s)^\ast$, $s\in S$.
		\item	$\rho(t)\rho(s)\subseteq \rho(st)$ with 
			$\cD(\rho(t)\rho(s))=\cD(\rho(st))\cap \cD(\rho(s))$, $s,t\in S$.
		\item	$\overline{\rho(t)\rho(s)}=\rho(st),$ $s,t\in S$.
		\item $D_c:=\bigcup_{\alpha\in \cA(S)}D_\alpha$ is dense in $\cH$.
	\eenu
\cite[Theorem 1.2.]{Ressel-Ricker2002} is the following result.

\begin{theorem} \label{polgrupe-neomejen}
		Assume the notation above. Let $\rho:S\to \cN(\cH)$ be a $\ast$-rep\-re\-sen\-tation.
	Then there exists a unique regular spectral measure $E:\B(S^\ast)\to B(\cH)$,
	such that $\supp(E_{h,h})$ is compact iff $h\in \cD_c$
	and
		$\rho(s)x=\int_{S^\ast}f_s(\chi)\;dE(\chi)\;x$
	for every $x\in \cD(\rho(s))$ and all $s\in S$.
\end{theorem}
 
\section{Proofs of Theorems A and B} \label{solutions-to-A-B}

	%The following solution of Problem A easily follows from Theorem \ref{polgrupe-omejen}.
%
%\begin{theorem} \label{theorem-A-1}
		%Let $\cB$ be a commutative $\ast$-algebra and $\cW$ a von Neumann algebra. 
	%Let $\widehat{\cB}$ be the character space of $\cB$ and
	%$\Bor(\widehat{\cB})$ a Borel $\sigma$-algebra on $\widehat{\cB}$.
	%The following statements are equivalent.
		%\begin{enumerate}
			%\item $\rho: \cB \rightarrow \cW$ is a unital $\ast$-representation.
			%\item There exists a unique regular spectral measure 
				%$E : \B(\widehat{\cB}) \rightarrow \cW$ with a compact support such that 
					%$$\rho(b) =	\int_{\widehat{\cB}} f_b(\chi)\; dE(\chi)$$ 
				%for all $b \in \cB$.
		%\end{enumerate}
%\end{theorem}

\begin{proof}[Proof of Theorem \ref{theorem-A}]
		The non-trivial direction is $(1)\Rightarrow (2)$. In particular, $\cB$ is a commutative semigroup 
	with an involution. By Theorem \ref{polgrupe-omejen}, there exists a unique regular
	spectral measure $\tilde E : \B(\cB^{\ast}) \rightarrow \cW$ with a compact support such 
	that 
	$\rho(b) =	\int_{\cB^\ast} f_b(\chi)\; d\tilde E(\chi)$ for every $b \in \cB$. 
	By the proof of \cite[Theorem 1]{Ressel-Ricker1998}, since $\cB$ is induced from an algebra,
	the support of $\tilde E$ is contained in the set of linear characters of $\cB$. 
	Hence, $E : \B(\widehat{\cB}) \rightarrow \cW$, defined by $E(\Delta)=\tilde E(\Delta)$,
	satisfies the statement of Theorem \ref{theorem-A}.
\end{proof}

%\begin{theorem} \label{theorem-B-1} 
		%Let $\cB$ be a commutative $\ast$-algebra and $\cW_1\subseteq B(\cH)$, $\cW_2\subseteq B(\cK)$ 
	%von Neumann algebras, where $\cH$, $\cK$ are Hilbert spaces. 
	%Let $\widehat{\cB}$ be the character space of $\cB$ and	$\B(\widehat{\cB})$ a Borel $\sigma
	%$-algebra on $\widehat{\cB}$.
	%The following statements are equivalent.
		%\begin{enumerate}
			%\item $\rho: \cB \otimes \cW_1\rightarrow \cW_2$ is a unital $\ast$-representation.
			%\item There exists a unique regular normalized non-negative spectral measure 
				%$M : \B(\widehat{\cB}) \rightarrow B(\cW_1,\cW_2)$ with a compact support such that 
					%$$\rho(F) =	\int_{\widehat{\cB}} f_F(\chi)\; dE(\chi)$$ 
				%for every $F \in \cB \otimes \cW_1$.
		%\end{enumerate}
%\end{theorem}
Assume the notation as in Theorem \ref{theorem-B}.
For every $b\in \cB$, we define the linear operator
		$\rho_b:\cW_1\to \cW_2$ by $\rho_b(A):=\rho(b\otimes A)$. 
We notice the following.

%Recall from the Introduction, that 
		%$f_{\sum_{j=1}^m b_j\otimes A_j}(\chi):=\sum_{j=1}^m f_{b_j}(\chi)\otimes A_j.$
%To prove Theorem \ref{theorem-B-1} we first notice, that the linear operators  
		%$$\rho_b:\cW_1\to \cW_2,\quad \rho_b(A):=\rho(b\otimes A)$$ 
%are bounded.

\begin{proposition} \label{trditev-o-omejenosti}
	$\rho_b$ is a bounded linear operator for every $b\in \cB$.
\end{proposition}

\begin{proof}
		Let us write $b=\frac{(b+1)^\ast(b+1)}{2}-\frac{(b-1)^\ast(b-1)}{2}$. Denote $c=\frac{(b+1)(b+1)^\ast}{2}$ and $d=\frac{(b-1)(b-1)^\ast}{2}$. 
	Then $\rho_b=\rho_c-\rho_d$. Thus it suffices to prove, that $\rho_b$ is bounded for every $b\in \cB^2:=\left\{a^\ast a\colon a\in \cB\right\}$. 
	Since $\rho$ is a $\ast$-representation, $\rho_b:\cW_1\to \cW_2$ is a positive linear operator for every $b=a^\ast a\in \cB^2$. Indeed, 
		\begin{eqnarray*}
			\rho_b(A)&=& \rho(b\otimes A)=\rho(a^\ast a\otimes A^{\frac{1}{2}}A^{\frac{1}{2}})=
									 \rho((a\otimes A^{\frac{1}{2}})^\ast (a\otimes A^{\frac{1}{2}}))\\
							 &=& \rho((a\otimes A^{\frac{1}{2}})^\ast) \rho(a\otimes A^{\frac{1}{2}})
							 =   \rho(a\otimes A^{\frac{1}{2}})^\ast \rho(a\otimes A^{\frac{1}{2}})\in (\cW_2)_+
		\end{eqnarray*}
	where $A^{\frac{1}{2}}$ is a positive square root of $A$. 
	For every $k$ in the Hilbert space $\cK$, where $\cW_2\subseteq B(\cK)$,
		$(\rho_b)_{k}:\cW_1\to \CC,$ defined by $(\rho_b)_{k}(A):=\left\langle \rho(b\otimes A)k,k \right\rangle$,
	is a positive linear functional. By \cite[5.12. Corollary]{Con}, it is 
	bounded and has norm $(\rho_b)_k(\id_{\cH})=\left\langle \rho(b\otimes \id_{\cH})k,k \right\rangle$. 
	Hence, 
		$\left|(\rho_b)_k(A)\right|\leq (\rho_b)_k(\id_{\cH}) \|A\|\leq \|\rho(b\otimes \id_{\cH})\| \|k\|^2 \|A\|$
	for every $A\in \cW_1$.
	Therefore, 
		$$\|\rho_b(A)\|=\sup_{\|k\|=1}\left|(\rho_b)_k(A)\right|\leq 
			\sup_{\|k\|=1}\|\rho(b\otimes \id_{\cH})\| \|k\|^2 \|A\|
			=\|\rho(b\otimes \id_{\cH})\|\|A\|$$ 
	for every $A\in \cW_1$. 
	Hence, $\rho_b$ is a bounded linear operator with norm $\|\rho(b\otimes \id_{\cH})\|$.
\end{proof}

	Essential technical lemma in the proof of Theorem \ref{theorem-B} is the following.

\begin{lemma} \label{gostost-krogle}
	%	Let $\cB$ be a commutative $\ast$-algebra. Equip the character space $\widehat{\cB}$ with a Borel $\sigma$-algebra $\Bor(\widehat{\cB})$. 
	Let $K\in \Bor(\widehat{\cB})$ be a compact set. Let us equip the algebra
		$\cA_K=\left\{(f_b)|_K\in C(K)\colon b\in \cB\right\}$
of continuous functions with a supremum norm. Suppose $M(K)$ is the normed space of all complex-valued regular Borel measures on $K$ equipped with a variation norm. Then the unit ball 
		$B_{\cA_K}:=\left\{f\in \cA_K\colon \|f\|_{\infty}\leq 1\right\}$
	is dense in the unit ball of $C(K)^{\ast\ast}=M(X)^\ast$ equipped with a 	
	weak$^\ast$-topology.
\end{lemma}

\begin{proof}
	By \cite[V.4.1. Proposition]{Con}, the unit ball of $C(K)$ is dense in $M(X)^\ast$ equipped with a weak$^\ast$-topology. By \cite[p.\ 95]{B-M-K}, 
	$B_{\cA_K}$ is dense in the ball of $C(K)$.
\end{proof}

%us define the algebra 
%		Let $X$ be a compact Hausdorff space. Equip the algebra of continuous complex function $C(X,\CC)$ with a supremum norm $\|\cdot\|_{\infty}$.
%	Suppose $M(X)$ is the normed space of all complex-valued regular Borel measures on $X$ equipped with a variation norm. Then the unit ball 
%	$B_{C(X,\CC)}:=\left\{f\in C(X,\CC)\colon \|f\|_{\infty}\leq 1\right\}$ is dense in the unit ball of $C(X,\CC)^{\ast\ast}=M(X)^\ast$ equipped with a 	
%	weak$^\ast$-topology.

\begin{proof}[Proof of Theorem \ref{theorem-B}]
	%Direction $(2)\Rightarrow (1).$ Since $M$ has a compact support, $f_F$ is $M$-integrable for every 
%$F\in \cB\otimes \cW_1$. For $\rho$ defined by (2), we have to prove the linearity,
%the multiplicativity and the equality
	%$\rho(F^\ast)=\rho(F)^{\ast}$ for every $F\in \cB\otimes \cW_1$. The linearity follows by 
%\cite[Proposition 3.5]{Zalar2014}, while the multiplicativity by 
%\cite[Proposition 7.2]{Zalar2014}. To show 
%$\int_{\cB} f_{F}^\ast\; dM=(\int_{\cB} 	f_F\;dM)^{\ast}$  
	%it suffices, by the linearity, to consider elements of the form $F=b\otimes A$, $A\in (\cW_1)_+$. Since $M_A$ is a positive operator-valued measure, we have
	%$\int_\cB (f_b\otimes A)^\ast\; dM=\int_\cB (f_b\otimes A)\;dM=(\int_\cB (f_b\otimes A)\;dM)^{\ast}$ and the result follows.
	The non-trivial direction is $(1)\Rightarrow (2).$ Since $\rho$ is a $\ast$-rep\-re\-sen\-tation, the maps
		$\rho_P: \cB \to \cW_2,$ $\rho_P(b):=\rho(b\otimes P)$ are $\ast$-representations 
	for every $P\in (\cW_1)_p$.	
	By Theorem A, there exist unique spectral measures
	$E_P:\B(\widehat{B})\to \cW_2$ with a compact support, 
	such that $\rho_P(b)=\int_{\widehat{B}} f_b(\chi)\; dE_P(\chi)$ holds for every $b\in \cB$ 
	and every $P\in (\cW_1)_p$. 
	The idea is to show that the family
	$\{E_P\}_{P\in (\cW_1)_p}$ satisfies the conditions of Theorem \ref{karakterizacija-nenegativnih-spektralnih-mer} to obtain a non-negative spectral measure $M$ representing $\rho$.
	
	First let us show, that $\supp(E_P)\subseteq \supp(E_{\id_{\cH}})$ for every $P\in (\cW_1)_p$.
	Define $K:=\supp(E_P)\;\cup\; \supp(E_{\id_{\cH}})$. 
	Let us assume that $\supp(E_P)\setminus \supp(E_{\id_{\cH}})\neq \emptyset$ for
	some $P\in (\cW_1)_p$. Hence, 
	there are a set $\Delta\in \supp(E_P)\setminus \supp(E_{\id_{\cH}})$ and a vector $k_1\in \cK$ 
	such that $(E_P)_{k_1,k_1}(\Delta)>0$ and $(E_{\id_{\cH}})_{k_1,k_1}(\Delta)=0$.
	By Lemma \ref{gostost-krogle}, there exists a net $f_{b_j}\in C(K)$, such that 
	$\int_K f_{b_j}\; d(E_{P})_{k_1,k_1} \to (E_P)_{k_1,k_1}(\Delta)>0$ and 
	$\int_K f_{b_j}\; d(E_{\id_{\cH}-P})_{k_1,k_1} \to (E_{\id_{\cH}-P})_{k_1,k_1}(\Delta)\geq 0.$
		%\begin{eqnarray*}
		%	\int_K f_{b_j}\; d(E_{P})_{k_1,k_1} &\to& (E_P)_{k_1,k_1}(\Delta)>0,\\
		%	\int_K f_{b_j}\; d(E_{\id_{\cH}-P})_{k_1,k_1} &\to& (E_{\id_{\cH}-P})_{k_1,k_1}(\Delta)\geq 0,
		%\end{eqnarray*}
			 %\; \int_K f_{b_j}\; d(E_{\id_{\cH}})_{k_1,k_1} \to (E_{\id_{\cH}})_{k_1,k_1}(\Delta)=0.$$
	However, for every $j$ in the net we have
		\begin{eqnarray*}
			0&=&\int_K f_{b_j}\; d(E_{\id_{\cH}})_{k_1,k_1}=
			\left\langle \rho(b_j\otimes \id_{\cH})k_1,k_1\right\rangle=\\
			&=& \left\langle \rho(b_j\otimes P)k_1,k_1\right\rangle +
			\left\langle \rho(b_j\otimes (\id_{\cH}-P))k_1,k_1\right\rangle=\\
			&=&	\int_K f_{b_j}\; d(E_{P})_{k_1,k_1} + \int_K f_{b_j}\; d(E_{\id_{\cH}-P})_{k_1,k_1}.
		\end{eqnarray*}
	Since $\lim_j\int_K f_{b_j}\; d(E_{P})_{k_1,k_1}>0$ and 
	$\lim_j\int_K f_{b_j}\; d(E_{\id_{\cH}-P})_{k_1,k_1}\geq 0$, this is a contradiction.
	
		It remains to check first that the family $\{E_P\}_{P\in (W_1)_p}$ satisfies the conditions of Theorem \ref{karakterizacija-nenegativnih-spektralnih-mer}, second that $M$ is a representing measure of $\rho$ and finally that $M$ is unique, regular and normalized. Since the support of every $E_P$ is contained in the compact set $K=\supp(E_{\id_{\cH}})$, the proofs are the same as the proof of direction 
$(1)\Rightarrow(2)$ of \cite[Theorem 9.1]{Zalar2014}, just that we replace the use of \cite[Lemma 2.3]{Zalar2014} by Lemma \ref{gostost-krogle} above.
\end{proof}

\section{Proof of Theorem C'} \label{sekcija-problema-C}

A \textsl{bounding sequence} $E_n$ of a closed operator $A$ in a Hilbert space $\cK$ 
is an increasing family of hermitian projections, such that $\cup_{n=1}^{\infty} E_n=\id_{\cK}$
and $E_nA\subseteq AE_n$ and $AE_n$ is bounded everywhere defined operator on $\cK$ for every $n\in\NN$.

\begin{proof}[Proof of Theorem \ref{theorem-C'}]
	The direction $(2)\Rightarrow (1)$ follows by Propositions \ref{integrabilnost} and \ref{gostost-D-ja}.
	Now we will prove the direction $(1)\Rightarrow (2)$.
	In particular, $\cB$ is a commutative semigroup with an involution, where the group operation
	is the algebra multiplication. Let $\cN(\cH)$ be the set of normal operators on a Hilbert space $\cH$.
	First we prove, that the representation 
	$\rho_1:\cB\to \cN(\cH)$, $\rho_1(b):=\overline{\rho(b)}$, 
	is well-defined and satisfies the conditions (1)-(5) in the definition of a $\ast$-representation of a 
	semigroup with an involution as defined in Subsection \ref{subsection2}.
	
	 \textbf{Well-definedness:} 
	$\rho_1(b)$ has to be a normal operator for every $b\in \cB$. Since $\rho$ is integrable, this is true by \cite[Theorem 9.1.2]{Sch2}.
	
	 \textbf{Condition (1):} By the definition of a $\ast$-representation,
	$\rho(1)=\id_{\cD}$ and hence $\rho_1(1)=\id_{\cH}$. 
	
	 \textbf{Condition (2):} By the definition of integrability,	
	$\overline{\rho(b^\ast)}=\rho(b)^\ast$. Since $\rho(b)$ is closable, 
	$\rho(b)^\ast=\overline{\rho(b)}^\ast$. Thus, $\rho_1(b^\ast)=\rho_1(b)^\ast$.
	
	 \textbf{Condition (3):}
	The following chain is true.
		\begin{eqnarray*}
			\overline{\rho(a)}\;\overline{\rho(b)} &=&
			\overline{\rho(a^{\ast\ast})}\;\overline{\rho(b^{\ast\ast})}=
			\rho(a^\ast)^\ast\rho(b^\ast)^\ast\subseteq
			(\rho(b^\ast)\rho(a^\ast))^\ast\\
			&=& \rho(b^\ast a^\ast)^\ast=
			\overline{\rho((b^\ast a^\ast)^\ast)}=
			\overline{\rho(ab)},
		\end{eqnarray*}
	where the second equalities in both lines follow by integrability of $\rho$.
	Hence, $\rho_1(a)\rho_1(b)\subseteq \rho_1(ab)$.
	and 
	$\cD(\rho_1(a)\rho_1(b)) \subseteq \cD(\rho_1(b)) \cap \cD(\rho_1(ab))$. 
	To prove the converse inclusion, let
	$h\in \cD(\rho_1(b)) \cap \cD(\rho_1(ab))$. 
	By the integrability of $\rho$, there is an abelian von Neumann algebra $\cN$ such that $\rho_1(a)$ 
	is affiliated with $\cN$ for all $a\in \cB$ (see \cite[Theorem 9.1.7]{Sch2}). 
	By \cite[Theorem 5.6.15]{Kad-Rin}, $\rho_1(a)$, $\rho_1(b)$ and $\rho_1(ab)$ 
	have a common bounding sequence $E_n$. 
	Since $h\in \cD(\rho_1(b)),$ we have $E_n\rho_1(b)h=\rho_1(b)E_nh$.
	Hence, $\rho_1(a)\rho_1(b)E_nh=\rho_1(a)E_n(\rho_1(b)h)$ is well-defined. By,
	$\rho_1(a)\rho_1(b)\subseteq \rho_1(ab)$, we have 
	$\rho_1(ab) E_nh = \rho_1(a)\rho_1(b)E_n h$ for every $n\in \NN$.
	As $n\to \infty$, it follows that 
		$\rho_1(b)E_n h \to \rho_1(b)h$ (by $h \in \cD(\rho_1(b))$) and 
		$\rho_1(ab)E_nh \to \rho_1(ab)h$ (by $h \in \cD(\rho_1(ab))$). 
	Therefore, since the operator $\rho_1(a)$ is closed, it follows 
	that $\rho_1(b)h \in \cD(\rho_1(a))$. So, $h \in  \cD(\rho_1(a)\rho_1(b))$.
	Hence, $\cD(\rho_1(a)\rho_1(b))=
	\cD(\rho_1(ab))\cap \cD(\rho_1(b))$. Therefore, (iii) is true.
	
	 \textbf{Condition (4):}
	By (iii), we have $\overline{\rho_1(a)\rho_2(b)}\subseteq \rho_1(ab)$.
	We also know, that 
		$$\rho_1(ab)=\overline{\rho(ab)}=\overline{\rho(a)\rho(b)}\subseteq 
			\overline{\overline{\rho(a)}\;\overline{\rho(b)}}=\overline{\rho_1(a)\rho_1(b)}.$$
	Thus, $\rho_1(ab)\subseteq \overline{\rho_1(a)\rho_2(b)}$.
	Hence, $\overline{\rho_1(a)\rho_1(b)}=\rho_1(ab)$. This is exactly the condition (4).
	
	 \textbf{Condition (5):} It is fulfilled by the assumption.
	
	Thus, by Theorem \ref{polgrupe-neomejen} above,	there exists a unique regular spectral measure
	$E:\Bor(\cB^\ast)\to B(\cK)$ such that 
	$\overline{\rho(b)}x=\int_{\cB^\ast}f_b(\chi)\;dE(\chi)\;x,$
	for every $b\in \cB$ and every $x\in \cD(\overline{\rho(b)})$.
	Since $\cB$ is also an algebra, $E$-almost every character $\chi\in \cB^\ast$ is linear (see 
		\cite[Proof of Theorem 1, p.~2953]{Ressel-Ricker1998} 
	and
		\cite[Proof of Theorem 1.2, p.~230]{Ressel-Ricker2002}).
	Therefore, $\overline{\rho(b)}x=\int_{\widehat\cB}f_b(\chi)\;dE(\chi)\;x$
	for every $x\in \cD(\overline{\rho(b)})$.
\end{proof}

\section{Integral of unbounded functions with respect to a non-negative spectral measure}
\label{razsiritev-integracije}

Let $(X,\Bor(X),\cW\subseteq B(\cH),B(\cK),M)$ be a space with a non-negative spectral measure $M$,
where $X$ is a topological space, $\Bor(X)$ a Borel $\sigma$-algebra on $X$,
$\cW\subseteq B(\cH)$ a von Neumann algebra
and $B(\cH), B(\cK)$ the bounded linear operators on Hilbert spaces $\cH$, $\cK$.

Before solving Problem D, the definition of the integration of functions with respect to a non-negative spectral measure $M:\Bor(X)\to B(\cW,B(\cK))$ has to be extended to a larger set of functions. In \cite{Zalar2014}, a measurable function $f$ was $M$-integrable, if it was $M_A$-essentially bounded for every positive operator $A\in \cW_+$, i.e., there exists a constant $K_f$ such that
$\left\langle (\int_X \left|f\right|\;dM_A)k,k\right\rangle\leq K_f \|k\|^2$ for every $k\in \cK$. In this section we extend this definition to functions, which do not necessarily satisfy this condition. 

Let $\scrK:=\left\{K\subseteq X\colon K \text{ compact}\right\}$ be the set of all compact subsets of $X$. Let us define the set $\cD_0\subseteq \cK$ by
		$$\cD_0:=\bigcup_{K\in \scrK}\; M(K)(\id_{\cH})\cK.$$

\begin{proposition} \label{gostost}
	$\cD_0$ is a dense linear subspace in $\cK$ in either of the following cases:
		\begin{itemize}
			\item The topological space $X$ is $\sigma$-compact.
			\item The measure $M$ is regular.
		\end{itemize}
\end{proposition}

\begin{proof}
	First we show $\cD_0$ is a linear subspace in $\cK$.
	Let $x_1, x_2\in \cD_0$ and $\lambda_1, \lambda_2\in \CC$. 
	Then it holds that $x_1=M(K_1)(\id_{\cH})k_1$, $x_2=M(K_2)(\id_{\cH})k_2$
	for some compact sets $K_1, K_2$ and some vectors $k_1, k_2\in \cK$. Since 
	$M(\Delta)(\id_{\cH})$ is a hermitian projection for every $\Delta\in \Bor(X)$ 
	it follows that 
		$M(K_i)(\id_{\cH})x_i=(M(K_i)(\id_{\cH}))^2k_i=M(K_i)(\id_{\cH})x_i=x_i$
	for $i=1,2.$ By the inequality 
	$M(\Delta_1)(\id_{\cH})\preceq M(\Delta_2)(\id_{\cH})$ if $\Delta_1\subseteq \Delta_2$, it follows that
		$\lambda_1 x_1+\lambda_2 x_2=M(K_1\cup K_2)(\id_{\cH})(\lambda_1 x_1+\lambda_2 x_2).$
	Hence, $\lambda_1 x_1+\lambda_2 x_2\in \cD_0$ and
	$\cD_0$ is a linear subspace in $\cK$.
	
	Now we prove $\cD_0$ is dense in $\cK$. We separate two cases:
	\begin{itemize}
		\item Let us assume $X$ is $\sigma$-compact. Take $k\in \cK$. By $M(X)(\Id_{\cH})k=k$ and by the 
			$\sigma$-compactness of $X$, we have 
				$k=M(X)(\Id_{\cH})k=\displaystyle\lim_{n\in \NN} M(K_n)(\Id_{\cH})k,$
			where $K_n$ is an increasing sequence of compact sets, such the $\cup_{n} K_n=X$.
			Hence, $\cD_0$ is dense. 
		\item Let now $X$ be arbitrary and $M$ regular.
			Since $(M_{\id_{\cH}})_{k,k}$ is regular for every $k\in \cH$, 
				$\displaystyle\sup_{K\in\scrK} (M_{\id_{\cH}})_{k,k}(K)=
				 (M_{\id_{\cH}})_{k,k}(X)=\|k\|^2$. Then 
				\begin{eqnarray*}
					&&\|k-M_{\id_{\cH}}(K)k\| = \|M_{\id_{\cH}}(K^c)k\|=
					((M_{\id_{\cH}})_{k,k}(K^c))^{\frac{1}{2}}\\
					&=&	((M_{\id_{\cH}})_{k,k}(X)-(M_{\id_{\cH}})_{k,k}(K)
						)^{\frac{1}{2}}
					= (\|k\|^2-(M_{\id_{\cH}})_{k,k}(K))^{\frac{1}{2}},
				\end{eqnarray*}
			is true for every $K\in \scrK$ and hence, $\displaystyle\sup_{K\in \scrK} M_{\id_{\cH}}(K)k=k$
			\qedhere
		\end{itemize}
\end{proof}

%Fix a dense subspace $\cD$ of $\cK$. We say a function $f$ is \textsl{$M$-integrable on $\cD$} if 
%$\cD$ is in the domain of the operator $\int_X f\; dM_P$ for every hermitian projection $P\in B(\cH)_p$.
%Let $\cI(M,\cD)$ be the set of all $M$-integrable functions on $\cD$.
%As in \cite{Zalar2014}, we define a map 
	%$\cB:I(M,\cD)\times B(\cH)\to L(\cD,\cK).$
%With almost the same proof as for \cite[Proposition 3.4]{Zalar2014}, the map $\cB$ is bilinear. 
%Hence, it extends to the linear map $\overline{\cB}$ on the tensor product $I(M,\cD)\otimes B(\cH)$.

%In this section we construct an integral of an unbounded measurable function with respect to a non-negative spectral measure by mimicking the construction of an integral with respect to a usual
%spectral measure (see Subsection \ref{subsection1}).

From now on we will assume that the space $\cD_0$ is dense. We say a $\Bor(X)$-measurable function $f$ is \textsl{$M$-integrable} if $\cD_0$ is in the domain of the operator $\int_X f\; dM_P$ for every hermitian projection $P\in (\cW)_p$. By $\cI(M)$ we denote be the set of all 
$M$-integrable function. 

\begin{proposition}
	The set $\cI(M)$ is a complex vector space and it contains all bounded functions on $X$ 
	and all continuous functions on $X$.
\end{proposition}

\begin{proof}
	The fact that $\cI(M)$ is a complex vector space is clear. Let now $f$ be a bounded function on $X$ 
	or a continuous functions on $X$. 
	Pick $k\in \cD_0$. 
	Then there are a compact set $K\in \scrK$ and 
	$k'\in \cK$ such that $M(K)(\id_\cH)k'=k$. Hence, 
	$M(K)(\id_\cH)k=M(K)(\id_\cH)^2k'=M(K)(\id_\cH)k'=k$. Thus, by the compactness of $K$ and 
	$(M_{\id_{\cH}})_{k,k}$ being a finite measure, it holds that
		$\int_{X}\left|f\right|^2\; d(M_{\id_{\cH}})_{k,k}
			=	\int_{K}\left|f\right|^2\; d(M_{\id_{\cH}})_{k,k}<\infty.$
		%\begin{eqnarray*}
		%	\int_{X}\left|f(x)\right|^2\; d(M_{\id_{\cH}})_{k,k}(x)
		%	&=&	\int_{K}\left|f(x)\right|^2\; d(M_{\id_{\cH}})_{k,k}(x)<\infty
		%\end{eqnarray*}
	Since $(M_{\id_{\cH}})_{k,k}=(M_{P})_{k,k}+(M_{\id_{\cH}-P})_{k,k}$, $k$ also belongs to
	$\II_{M_P}(f)$ for every hermitian projection $P\in \cW_p$. Hence, $\cD_0\subseteq \II_{M_P}(f)$.
\end{proof}

\subsection*{Construction of an integral $\mathbb{I}_M$} Now we explain the construction of the integral
with respect to a non-negative spectral measure, denoted by $\mathbb{I}_M$, and prove well-definedness in what follows.

\noindent \textbf{Step 1 -} \textsl{The construction of a preintegral $\psi:I(M)\times \cW\to \cL^+(\cD_0)$:}

\begin{itemize} 
	\item For $(f,P)\in \cI(M)\times \cW_p$, we define 
		$\psi(f,P):=(\int_X f\;dM_p)|_{\cD_0}.$
	\item For $(f,A)\in \cI(M)\times \cW_+$ such that $A$ has a finite spectral decomposition
		$\sum_{k=1}^n\lambda_k P_k$, where $\lambda_k\geq 0$ are non-negative and
		$P_k$ are mutually orthogonal hermitian projections (i.e., $P_iP_j=0$ for every $i\neq j$),
		we define
			$\psi(f,A):= \sum_{k=1}^n\lambda_k \psi(f, P_k).$
	\item For $(f,A)\in \cI(M)\times \cW_+$ such that $A$ does not have a finite spectral decomposition,
		take an arbitrary limiting sequence $S_\ell(A)$ and define
			$\displaystyle\psi(f,A):=\lim_{\ell\to\infty} \psi(f,S_\ell(A)).$
		By Proposition \ref{proposition3} below, $\psi(f,A)$ is well-defined and does not depend
		on the choice of the sequence $S_\ell(A)$.
	\item For $(f,A)\in \cI(M)\times \cW$,
		define
		$$\psi(f,\re(A)_+)x-\psi(f, \re(A)_-)x+ 
			 i\cdot \psi(f, \im(A)_+)x-
			 i\cdot \psi(f, \im(A)_-)x,$$
		where $\re(A),\im(A)$ denote the real and the imaginary part of the operator $A$ and
		$A_+, A_-$ the positive and the negative part of the hermitian
		operator $A$.
	\item $\psi$ is a bilinear form (see Proposition \ref{bilinearnost-psija}).
\end{itemize}

\noindent \textbf{Step 2 -} \textsl{Defining a map $\tilde\psi:\cI(M)\otimes \cW\to \cL^+(\cD_0)$:}

Since $\psi$ is a bilinear form, it extends, by the universal property of tensor products, to
the linear map 
	$\overline{\psi}:\cI(M) \otimes \cW \to \cL^+(\cD_0)$, defined by
		$\overline{\psi}(\sum_{i=1}^{n}f_i\otimes A_i)=
		\sum_{i=1}^{n}\psi(f_i, A_i).$

\noindent \textbf{Step 3 -} \textsl{Defining an integral $\mathbb{I}_M$:}

We extend $\tilde{\psi}$ to \textsl{the integral} $\mathbb{I}_M:\cI(M) \otimes \cW \to \cN(\cK)$, defined by the rule
	$\mathbb{I}_M(\sum_{i=1}^{n}f_i\otimes A_i)=
		\overline{\tilde{\psi}(\sum_{i=1}^{n}f_i\otimes A_i)}$,
where $\cN(\cK)$ denotes the vector space of normal operators on $\cK$ 
and $\overline{T}$ denotes the closure of a densely defined operator $T\in \cL^+(\cD_0)$. By Proposition \ref{dobra-definiranost-I-ja},
$\mathbb{I}_M$ is well-defined.

It remains to prove Propositions \ref{proposition3}, \ref{bilinearnost-psija}
and \ref{dobra-definiranost-I-ja} which were needed for the construction of $\mathbb{I}_M$. 
First we prove the following lemma, which will be needed in the proofs.

%Let $\cU_b$ be the set of all measurable function $f$, which are bounded on compact subsets $\Bor(X)$.

\begin{lemma}\label{lema5}
	Let $f\in \cI(M)$. Then $\cD_0$ is 
	%contained in the domain $\cD(\II_{M_P}(f))$ and is 
	a core for $\II_{M_P}(f)$. Also,
		$$M(K)(\id_{\cH})\II_{M_P}(f)\subseteq \II_{M_P}(f)M(K)(\id_{\cH})=
			\II_{M_P}(f\chi_{K})$$
	for every compact set $K$.
\end{lemma}

\begin{proof}
	%Let $K=\Delta_1$ be a compact set and $x\in\cK$. 
	%%Under the conditions of Proposition 5.1, we can
	%%By the $\sigma$-compactness of $X$ we can find a sequence $K_n$ of compact sets, such
	%%that $K_1:=K$, $K_n\subseteq K_{n+1}$ for $n\in\NN$ and $\cup_{n\in\NN}K_n=X$.
	%Let $\Delta_n$ be the bounding sequence of $f$.
	%By the boundedness of the function $f\chi_{\Delta_1}\chi_{\Delta_n}$ and by
	%\cite[Proposition 7.2]{Zalar2014}, we have
		%\begin{eqnarray*}
			%(\int_{X}{f\chi_{\Delta_1}\otimes P}\;dM)x 
				%&=&(\int_{X}{f\chi_{\Delta_1}\chi_{\Delta_n}\otimes P}\;dM)x\\
				%&=&(\int_{X}{f\chi_{\Delta_n}\otimes P}\;dM)M(\Delta_1)(\id_{\cH})x\\
				%&=&M(\Delta_1)(\id_{\cH})(\int_{X}{f\chi_{\Delta_n}\otimes P}\;dM)x.
		%\end{eqnarray*}
  %Hence, 
		%$\sup_{n\in\NN}\|\II_{M_P}(f\chi_{\Delta_n})M(\Delta_1)(\id_{\cH})x\|<\infty$
	%and by Theorem \ref{izrek3}.(i), $$M(\Delta_1)(\id_{\cH})x\in \cD(\II_{M_P}(f)).$$
	%That is, $\cD_0\subseteq \cD(\II_{M_P}(f))$.
	Let $x\in \cK$ be arbitrary. Let $\Delta_n$ be the bounding sequence of $f$.
	By the boundedness of the function $f\chi_{\Delta_m}\chi_{\Delta_n}$ and by
	\cite[Proposition 7.2]{Zalar2014}, we have the following chain of equalities 
	$(\int_{X}{f\chi_{\Delta_n}\otimes P}\;dM)x 
				=(\int_{X}{f\chi_{\Delta_n}\chi_{\Delta_m}\otimes P}\;dM)x
				=(\int_{X}{f\chi_{\Delta_m}\otimes P}\;dM)M(\Delta_n)(\id_{\cH})x
				=M(\Delta_n)(\id_{\cH})(\int_{X}{f\chi_{\Delta_m}\otimes P}\;dM)x,$
	where $m$ is greater than $n$.
			%\begin{eqnarray*}
			%(\int_{X}{f\chi_{\Delta_n}\otimes P}\;dM)x 
				%&=&(\int_{X}{f\chi_{\Delta_n}\chi_{\Delta_m}\otimes P}\;dM)x\\
				%&=&(\int_{X}{f\chi_{\Delta_m}\otimes P}\;dM)M(\Delta_n)(\id_{\cH})x\\
				%&=&M(\Delta_n)(\id_{\cH})(\int_{X}{f\chi_{\Delta_m}\otimes P}\;dM)x.
		%\end{eqnarray*}
	%where $m\geq n$. 
	Letting $m\to \infty$, we conclude that 
		$(\int_{X}{f\chi_{\Delta_n}\otimes P}\;dM)x
		=(\int_{X}{f	\otimes P}\;dM)M(\Delta_n)(\id_{\cH})x.$
	%$(\int_{X}{f\chi_{\Delta_n}\otimes P}\;dM)x
	%	=(\int_{X}{f	\otimes P}\;dM)M(\Delta_n)(\id_{\cH})x.$
	For $x\in \cD(\II_{M_P}(f))$, again sending $m\to \infty$, it follows that
		$(\int_{X}{f	\otimes P}\;dM)M(\Delta_n)(\id_{\cH})x
			=M(\Delta_n)(\id_{\cH})(\int_{X}{f\otimes P}\;dM)x.$
	Since $M(\Delta_n)(\id_{\cH})x\rightarrow x$ and 
		$\II_{M_P}(f)M(\Delta_n)(\id_{\cH})x=M(\Delta_n)(\id_{\cH})\II_{M_P}(f)x\rightarrow
		\II_{M_P}(f)x$
	for $x\in \cD(\II_{M_P}(f))$, the linear subspace 
		$\cup_{n=1}^{\infty} M(\Delta_n)(\id_{\cH})\cK$ is a core for $\II_{M_P}(f)$.
	By the $\sigma$-compactness of $X$ or regularity of $M$, we have a sequence of compact sets $K_n$,
	such that $\|M(\Delta_n)(\id_{\cH})x-M(K_n)x\|\leq \frac{1}{n\|f\|_{\Delta_n,\infty}}$ and 
	$\|(\int_{X}{f\chi_{\Delta_n}\otimes P}\;dM)x-
		(\int_{X}{f\chi_{K_n}\otimes P}\;dM)x\|
		=\|(\int_{X}{f\chi_{\Delta_n\setminus K_n}\otimes P}\;dM)x\|
		\leq \|f\|_{\Delta_n,\infty}\frac{1}{n\|f\|_{\Delta_n,\infty}}=\frac{1}{n}.$
	%\begin{eqnarray*}
		%&&\|(\int_{X}{f\chi_{\Delta_n}\otimes P}\;dM)x-
		%(\int_{X}{f\chi_{K_n}\otimes P}\;dM)x\|=\\
		%&=&\|(\int_{X}{f\chi_{\Delta_n\setminus K_n}\otimes P}\;dM)x\|
		%\leq \|f\|_{\Delta_n,\infty}\frac{1}{n\|f\|_{\Delta_n,\infty}}=\frac{1}{n}.
	%\end{eqnarray*}
	Hence, $M(K_n)x\to x$ and $(\int_{X}{f\chi_{K_n}\otimes P}\;dM)x\to \II_{M_P}(f)x$.
	Thus, $\cD_0$ is a core for $\II_{M_P}(f)$.
\end{proof}

\begin{proposition}\label{proposition3}
	For $f\in \cI(M)$ and a positive operator $A\in \cW_+$ without a finite spectral decomposition, the definition of 	
	$\psi(f,A)$:
		\benu
			\item is well-defined,
			\item does not depend on the choice of the sequence $S_\ell(A)$.
		\eenu
\end{proposition}

In the proof of Proposition \ref{proposition3} we need the following lemma.

\begin{lemma} \label{lema1}
	For $f\in \cI(M)$ and $P,Q\in \cW_p$ orthogonal hermitian projections it is true that:
	\benu
		\item $\psi(f, P+Q)=\psi(f,P)+\psi(f,Q).$
		\item $\im(\psi(f,P))\perp \im(\psi(f,Q)),$
	\eenu
	where $\im(T)$ denotes the image of the operator $T$.
\end{lemma}

\begin{proof}
	Since $P,Q$ are orthogonal hermitian projections, $P+Q$ is also a hermitian projection. 
	Since $M$ is a non-negative spectral measure,
	$\mathrm{Im}(M_P)$ and $\mathrm{Im}(M_Q)$ are orthogonal (Here  
	$\mathrm{Im}(M_P)$, $\mathrm{Im}(M_P)$ denote the images of $M_P$, $M_Q$, 
	i.e.,
		$\mathrm{Im}(M_P):=\bigcup_{\Delta\in\Bor(X)}M_P(\Delta)\cK$
		%\left\{x\in\cK\colon M_P(\Delta)x=x\; \text{for some}\; \Delta\in \Bor(X)\right\}
	and analoguously for $M_Q$.).
	Therefore, by the definition of $\cD(\II(f))$ (see Theorem \ref{izrek3}), 
		$\cD(\II_{M_{P+Q}}(f))=\cD(\II_{M_{P}}(f))\cap \cD(\II_{M_{Q}}(f)).$
	Let $\Delta_n$ be a bounding sequence of $f$ (w.r.t.\ $M_{\id_{\cH}}$ and hence all $M_P$).
	Since $f\chi_{\Delta_n}$ is a bounded measurable function,
		$\II_{M_{P+Q}}(f\chi_{\Delta_n})=\II_{M_{P}}(f\chi_{\Delta_n})+
		\II_{M_{Q}}(f\chi_{\Delta_n}),$
		 by \cite[Proposition 3.5]{Zalar2014}.
	Hence, by 
	\begin{eqnarray*}
		&&\II_{M_{P+Q}}(f)x 
			=\lim_{n\to\infty}\II_{M_{P+Q}}(f\chi_{\Delta_n})x=
				\lim_{n\to\infty}(\II_{M_{P}}(f\chi_{\Delta_n})x+\II_{M_{Q}}(f\chi_{\Delta_n})x)\\
			&=&\lim_{n\to\infty}\II_{M_{P}}(f\chi_{\Delta_n})x+
				\lim_{n\to\infty}\II_{M_{Q}}(f\chi_{\Delta_n})x
			=\II_{M_{P}}(f)x + \II_{M_{Q}}(f)x 
	\end{eqnarray*}
	for every $x\in \cD(\II_{M_{P+Q}}(f))$, it follows that
		$\II_{M_{P+Q}}(f)=\II_{M_{P}}(f)+\II_{M_{Q}}(f)$
	and by the definition of $\psi$ also 
		$\psi(f,P+Q)=\psi(f, P)+\psi(f,Q).$
	Since $M$ is a non-negative spectral measure, $M_P(\Delta)M_Q(\Delta')=0$ for every 
	$\Delta, \Delta'\in \Bor(X)$ such that $\Delta \cap \Delta'=\emptyset$, and hence also
		$\im(\psi(f,P))\perp \im(\psi(f,Q)).$
\end{proof}

\begin{proof}[Proof of Proposition \ref{proposition3}]
	Let us first prove that $(\psi(f,S_\ell(A))x)_{\ell\in\NN}$ is a Cauchy sequence.
	For $\ell',\ell\in\NN$, $\ell'>\ell$, we have
	$S_\ell(A)-S_{\ell'}(A)=\sum_{i=1}^{m_{\ell,\ell'}}\lambda_i P_i$ 
	for some $\lambda_i\in \RR$, mutually
	orthogonal hermitian projections $P_i$ and $m_{\ell,\ell'}\in\NN$. 
	Given $\epsilon>0$ and choosing $\ell$ great enough we can 
	achieve $\left|\lambda_i\right|<\epsilon$ for every $i=1,\ldots,m$. 
	Since $\id_{\cH}=P+(\id_{\cH}-P)$, where $P$, $\id_{\cH}-P$ are mutually orthogonal hermitian 
	projections, it follows that
		$\|\psi(f,P)x\|\leq \|\psi(f,\id_\cH)x\|$ 
	for every $x\in \cD_0$. We have
		\begin{eqnarray*}
			\|\sum_{i=1}^m\lambda_i\psi(f,P_i)x\|
				&\leq& \max_{i}\left|\lambda_i\right| \|\sum_{i=1}^m\psi(f,P_i)x\|
					\leq \max_{i}\left|\lambda_i\right| \|\psi(f,\id_{\cH})x\|\\
				&\leq& \epsilon \|\psi(f,\id_{\cH})x\|,
		\end{eqnarray*}
	where the first inequality follows by the fact that 
	$\im(\psi(f,P_i))\perp \im(\psi(f,P_j))$ for $i\neq j$ (see Lemma \ref{lema1}.(2))
	and the second by the fact that $\sum_{i=1}^m P_i$ is a hermitian projection.
	Since $\epsilon>0$ was arbitrary, $(\psi(f,S_\ell(A))x)_{\ell\in\NN}$ is a Cauchy sequence 
	and hence $\lim_{\ell\to \infty}\psi(f,S_\ell(A))x$ exists. This proves (1).
	
	Now we will prove the independence from the sequence $S_\ell(A)$. Let $S'_\ell(A):=\sum_{k=1}^{m_\ell}
	\zeta'_{k,l}P'_{k,l}$ be another sequence
	converging to $A$ in norm, where $\zeta'_{k,l}\geq 0$
	are non-negative and $P'_{k,l}$ are mutually orthogonal hermitian projections. We will prove that 
		$(\psi(f,S_\ell(A))x-\psi(f,S'_\ell(A))x)_{\ell\in\NN}$
	converges to $0$.
	We have
		$S_\ell(A)-S'_\ell(A)=\sum_{i=1}^{p}\mu_i Q_i,$
	where $\mu_i\in \RR$ and $Q_i$ are
	mutually orthogonal hermitian projections.
	Therefore
		$\psi(f,S_\ell(A))x-\psi(f,S'_\ell(A))x=\sum_{i=1}^{p}\mu_i \psi(f,Q_i)x.$
	Given $\epsilon>0$ and choosing $\ell$ great enough we can 
	achieve $\left|\mu_i\right|<\epsilon$ for every $i=1,\ldots,p$.
	As for part (i) we estimate
		$\|\sum_{i=1}^p\mu_i\psi(f,Q_i)x\|\leq \epsilon \|\psi(f,\id_{\cH})x\|.$
	Therefore $\psi(f,S_\ell(A))x-\psi(f,S'_\ell(A))x$ converges to $0$ which proves (2).
\end{proof}

\begin{proposition} \label{bilinearnost-psija}
	The map $\psi$ is bilinear.
\end{proposition}

To prove Proposition \ref{bilinearnost-psija} we need the following lemma.

\begin{lemma}\label{lema2}
	For a function $f\in \cI(M)$, an operator $A\in \cW$ and $x\in\cD_0$, we have
		$$\psi(f,A)x=\lim_{n\to\infty} \psi(f\chi_{\Delta_n},A)x,$$
	where $(\Delta_n)_{n\in\NN}$ is a bounding sequence of $f$ with respect to the spectral measure
	$M_{\id_{\cH}}$. %(see Subsection \ref{subsection1}).
	Furthermore, for every $n\in \NN$ and every $x\in\cD_0$ we have
		$$\psi(f\chi_{\Delta_n},A)x=\psi(f,A)M(\Delta_n)(\id_{\cH}).$$
\end{lemma}

\begin{proof}
	By the definition of $\psi(f,A)$ for an arbitrary operator $A\in \cW$, we may assume $A$ is
	a positive operator.
	For $f\in \cI(M)$, a positive operator $A\in \cW_+$ and $x\in\cD_0$, it holds by the definition that
	$\psi(f,A)x=\lim_{\ell\to\infty}\psi(f,S_\ell(A))x$, where
	$S_\ell(A)$ is a limiting sequence of $A$.
	By Theorem \ref{izrek3}, the equality $\psi(f,P)x=\lim_{n\to\infty}\psi(f\chi_{\Delta_n},P)x$
	holds for every hermitian projection $P\in \cW_p$, every bounding sequence $(\Delta_n)_n$ of $f$
	with respect to the spectral measure $M_{\id_{\cH}}$ and every $x\in \cD_0$.
	Using Lemma \ref{lema1} it is also true that for every hermitian projection $P\in \cW_p$ we have
		$$\|\psi(f,P)x-\psi(f\chi_{\Delta_n},P)x\|\leq
			\|\psi(f,\id_\cH)x-\psi(f\chi_{\Delta_n},\id_\cH)x\|,
		$$
	and hence, by an analogous estimate as in the proof of Proposition \ref{proposition3},
		\begin{eqnarray}\label{equality4}
			&&\|\psi(f,S_\ell(A))x-\psi(f\chi_{\Delta_n},S_\ell(A))x\|\leq
			\|A\|\|\psi(f,\id_\cH)x-\psi(f\chi_{\Delta_n},\id_\cH)x\|.
		\end{eqnarray}
	For every $\epsilon>0$ there exists $N\in \NN$ such that for every $n\geq N$
	we have 
		$$\|\psi(f,\id_\cH)x-\psi(f\chi_{\Delta_n},\id_\cH)x\|\leq 
			\frac{\epsilon}{3\|A\|}.$$
	There also exists $L_1(n)\in \NN$, such that for every $\ell\geq L_1(n)$ we have
		$$\|\psi(f,A)x-\psi(f,S_{\ell}(A)) x\|\leq 
			\frac{\epsilon}{3}.$$
	For every $n\in \NN$ there also exists $L_2(n)\in \NN$, such that for every $\ell\geq L_2(n)$
		$$\|\psi(f\chi_{\Delta_n},A)x-\psi(f\chi_{\Delta_n},S_{\ell}(A)) x\|\leq 
			\frac{\epsilon}{3}.$$
	Let $L(n)=\max\{L_1(n),L_2(n)\}$.
	Hence, for every $\epsilon>0$ there exists $N_\epsilon\in \NN$, such that for
	every $n\geq N_\epsilon$ we have
		\begin{eqnarray*}
			\|\psi(f,A)x-\psi(f\chi_{\Delta_n},A) x\|
			&\leq&
			\|\psi(f,A)x-\psi(f,S_{L(n)}(A)) x\|\\
			&+&\|\psi(f,S_{L(n)}(A))x-\psi(f\chi_{\Delta_n},S_{L(n)}(A)) x\|\\
			&+&	\|\psi(f\chi_{\Delta_n},S_{L(n)}(A))x-\psi(f\chi_{\Delta_n},A) x\|\\
			&\leq&\epsilon.
		\end{eqnarray*}
	Hence, $\psi(f,A)x=\lim_{n\to\infty}\psi(f\chi_{\Delta_n},A)x.$
	
	For the other part, 
	\begin{eqnarray*}
		\psi(f\chi_{\Delta_n},A)x &=&\lim_{\ell\to\infty}\psi(f\chi_{\Delta_n},S_\ell(A))x=
		\lim_{\ell\to\infty}\psi(f,S_\ell(A))M(\Delta_n)(\id_{\cH})x\\
		&=&\psi(f,A)M(\Delta_n)(\id_{\cH})x. \qedhere
	\end{eqnarray*}
\end{proof}

\begin{proof}[Proof of Proposition \ref{bilinearnost-psija}]
	We will first prove the linearity in the first factor. For $f, g\in \cI(M)$, 
	$\lambda,\mu\in \CC$ and $A\in \cW$,
	we have to show that $\psi(\lambda f+\mu g, A)=\lambda\psi(f,A)+\mu\psi(g,A)$. We may assume $A$ is 
	positive. Let $S_\ell(A)$ be a limiting sequence of $A$. For $x\in \cD_0$ it is true that
		\begin{eqnarray*}
			\psi(\lambda f+\mu g, A)x
				&=&\lim_{\ell\to\infty}\psi(\lambda f+\mu g, S_{\ell}(A))x=
					\lim_{\ell\to\infty}\sum_{k=1}^{n_\ell} \zeta_{k,l}\psi(\lambda f+\mu g,P_{k,l})x\\
				&=&\lim_{\ell\to\infty}\sum_{k=1}^{n_\ell} \zeta_{k,l}
					(\lambda \psi(f,P_{k,l})+\mu \psi(g,P_{k,l}))x\\
				&=&\lambda \lim_{\ell\to\infty} \psi(f,\sum_{k=1}^{n_\ell} \zeta_{k,l} P_{k,l})x
				 +\mu \lim_{\ell\to\infty} \psi(g,\sum_{k=1}^{n_\ell} \zeta_{k,l} P_{k,l})x\\
				&=&\lambda \psi(f,A)x + \mu \psi(g,A)x,
		\end{eqnarray*}
	where in the third equality we used the linearity of the integration with respect to the spectral measure
	$M_{P_{k,l}}$ (see Theorem \ref{izrek2}.(2)).
	
	Now we will prove the linearity in the second factor.
	We may assume $f\in\cI(M)$, $A,B\in \cW_+$ and $x\in \cD_0$.
	%we have to show that $\psi(f,\lambda A+ \mu B)=\lambda\psi(f,A)+\mu\psi(f,B)$.
	%By the usual decomposition of $\lambda,\mu,A,B$ into the linear combination of four positive parts
	%we may assume, $\lambda,\mu\geq 0$ and $A,B\in \cW_+$. Furthermore, we may assume that 
	%$\lambda=\mu=1$. For $x\in \cD_0$ and 
	Let $\Delta_n$ be a bounding sequence 
	%$\Delta_n$ 
	of $f$ with respect to the spectral measure $M_{\id_\cH}$. We have
	$\psi(f,A+B)x
				= \lim_{n\to\infty} \psi(f\chi_{\Delta_n},A+B)x
				= \lim_{n\to\infty} (\psi(f\chi_{\Delta_n},A)+\psi(f\chi_{\Delta_n},B))x
				= \psi(f,A)x+\psi(f,B)x,$
		%\begin{eqnarray*}
		%	\psi(f,A+B)x
		%		&=& \lim_{n\to\infty} \psi(f\chi_{\Delta_n},A+B)x
		%		= \lim_{n\to\infty} (\psi(f\chi_{\Delta_n},A)+\psi(f\chi_{\Delta_n},B))x\\
		%		&=& \psi(f,A)x+\psi(f,B)x,
		%\end{eqnarray*}
	where in the first equality we used Lemma \ref{lema2} and in the second equality we used 
	the linearity of integration of bounded functions with respect to the non-negative spectral measures
	(see \cite[Proposition 3.5.(3.1)]{Zalar2014}).
\end{proof}
	%For $f\in\cI(M)$, $\lambda,\mu\in\CC$ and $A,B\in \cW$,
	%we have to show that $\psi(f,\lambda A+ \mu B)=\lambda\psi(f,A)+\mu\psi(f,B)$.
	%By the usual decomposition of $\lambda,\mu,A,B$ into the linear combination of four positive parts
	%we may assume, $\lambda,\mu\geq 0$ and $A,B\in \cW_+$. Furthermore, we may assume that 
	%$\lambda=\mu=1$. For $x\in \cD_0$ and a bounding sequence $\Delta_n$ of $f$ with respect to
	%the spectral measure $M_{\id_\cH}$, it holds that
		%\begin{eqnarray*}
			%\psi(f,A+B)x
				%&=& \lim_{n\to\infty} \psi(f\chi_{\Delta_n},A+B)x
				%= \lim_{n\to\infty} (\psi(f\chi_{\Delta_n},A)+\psi(f\chi_{\Delta_n},B))x\\
				%&=& \psi(f,A)x+\psi(f,B)x,
		%\end{eqnarray*}
	%where in the first equality we used Lemma \ref{lema2} and in the second equality we used 
	%the linearity of integration of bounded functions with respect to the non-negative spectral measures
	%(see \cite[Proposition 3.5.(3.1)]{Zalar2014}).
%\end{proof}

\begin{proposition} \label{dobra-definiranost-I-ja}
	The map $\mathbb{I}_M$ is well-defined.
\end{proposition}

To proof Proposition \ref{dobra-definiranost-I-ja} we need  the following lemma.

\begin{lemma} \label{lema4}
	For $f\in \cI(M)$ and $A\in \cW$, we have
	$\cD_0\subseteq \cD(\overline{\psi}(f\otimes A)^{\ast})$ 
	and
		$$\overline{\psi}(f\otimes A)^{\ast}x=\overline{\psi}(\overline{f}\otimes A^\ast)x$$
	for every $x\in \cD_0$.
\end{lemma}

\begin{proof}
	By the decomposition of $A$ into the linear combination of four positive parts and since the domain 
	$\cD_0$ of $\cD(\overline{\psi}(f\otimes A))$ is dense in $\cK$, we may assume,
	by \cite[Proposition 1.6(vi)]{Sch2012}, that $A$ is a positive operator.
	For $x,y\in \cD_0$ and a limiting sequence $S_\ell(A)$ of $A$, it is true that
		\begin{eqnarray*}
			&&\left\langle\right. \overline{\psi}(f\otimes A)x,y\left.\right\rangle
				= \left\langle\right. \lim_{\ell\to\infty }\overline{\psi}(f\otimes S_{\ell}(A))x,y
					\left.\right\rangle
				 = \lim_{\ell\to\infty } \left\langle\right. \overline{\psi}(f\otimes S_{\ell}(A))x,y
					\left.\right\rangle\\
				&=& \lim_{\ell\to\infty } \left\langle\right. x,\overline{\psi}(f\otimes S_{\ell}(A))^\ast y
					\left.\right\rangle
				= \lim_{\ell\to\infty } \left\langle\right. x,\overline{\psi}(\overline{f}\otimes S_{\ell}(A)) y
				   \left.\right\rangle
				= \left\langle\right. x,\overline{\psi}(\overline{f}\otimes A) y\left.\right\rangle,
		\end{eqnarray*}
	where we used Theorem \ref{izrek2}.(i) and \cite[Proposition 1.6(vi)]{Sch2012} 
	in the fourth equality (Since the domain
		$\cD(\overline{\psi}(f\otimes S_\ell(A)))$ is dense,
		$$\overline{\psi}(f\otimes S_\ell(A))^\ast\supseteq 
			\sum_{i=k}^{n_\ell}\zeta_{k,l}\overline{\psi}(f\otimes P_{k,\ell})^\ast
			\supseteq \sum_{i=k}^{n_\ell}\zeta_{k,l}\overline{\psi}(\overline{f}\otimes P_{k,\ell})=
			\overline{\psi}(\overline{f}\otimes S_\ell(A)).)
			.$$ 
	Therefore 
	$y\in \cD(\overline{\psi}(f\otimes A)^\ast)$
	and $\overline{\psi}(f\otimes A)^{\ast}y=\overline{\psi}(\overline{f}\otimes A)y$.
\end{proof}

\begin{proof}[Proof of Proposition \ref{dobra-definiranost-I-ja}]
	For $\mathbb{I}_M$ to be well-defined,  
		$\overline{\psi}(\sum_{i=1}^{n}f_i\otimes A_i)$
	must be closable for every functions $f_1,\ldots,f_n\in \cI(M)$ and every $A_1,\ldots,A_n\in \cW$.
	By \cite[Theorem 1.8(i)]{Sch2012}, it suffices to show that the domain 
		$\cD((\overline{\psi}(\sum_{i=1}^{n}f_i\otimes A_i))^\ast)$
	is dense in $\cK$. Since the domain
		$\cD(\overline{\psi}(\sum_{i=1}^{n}f_i\otimes A_i))$ is dense,
	by \cite[Proposition 1.6(vi)]{Sch2012}, 
		$\overline{\psi}(\sum_{i=1}^{n}f_i\otimes A_i)^\ast\supseteq 
			\sum_{i=1}^{n}\overline{\psi}(f_i\otimes A_i)^\ast.$
	Therefore it suffices to show that $\sum_{i=1}^{n}\overline{\psi}(f_i\otimes A_i)^\ast$
	is dense\-ly defined. Furthermore, it suffices to prove that every 
	operator
	$\overline{\psi}(f\otimes A)^\ast$ with $f\in \cI(M)$ and $A\in B(\cH)$, is defined on $\cD_0$. 
	But this is the statement of Lemma \ref{lema4}, which concludes the proof.
\end{proof}

\subsection{Algebraic properties of $\mathbb{I}_M$} \label{algebraicne-lastnosti}

The main algebraic properties of the integral $\mathbb{I}_M$ are collected in the following theorem.

\begin{theorem} \label{izrek4}
	For $F, G\in \cI(M)\otimes \cW$, $\alpha, \beta\in \CC$, $f,g\in \cI(M)$ and a hermitian projection
	$P\in \cW_p$, we have:
		\benu
			\item $\mathbb{I}_M(F^\ast)\subseteq \mathbb{I}_M(F)^\ast,$
			\item $\mathbb{I}_M(\alpha F+ \beta G)= \overline{\alpha \mathbb{I}_M(F)+\beta 
				\mathbb{I}_M(G)},$
			\item $\mathbb{I}_M(FG)\subseteq \overline{\mathbb{I}_M(F)\mathbb{I}_M(G)},$
			%\item[(iv)] 
			%$\cD(\mathbb{I}_M(g\otimes P))\cap \cD(\mathbb{I}_M(fg\otimes P))\subseteq 
			%	\cD(\mathbb{I}_M(f\otimes P)\mathbb{I}_M(g\otimes P)).$
		\eenu
\end{theorem}

To prove Theorem \ref{izrek4} we need some additional lemmas.

\begin{lemma} \label{lema8}
	For every map $F\in \cI(M)\otimes \cW$ and every compact set $K$ we have
		$$M(K)(\id_{\cH})\overline{\psi}(F)\subseteq \overline{\psi}(F)M(K)(\id_{\cH})=
			\overline{\psi}(F\chi_{K}).$$
\end{lemma}

\begin{proof}
	%Let $K$ be a compact set. 
	By the linearity of $\overline{\psi}$ and the boundedness of $M(K)(\id_{\cH})$, we may assume
	$F=f\otimes A$, where $f\in\cI(M)$, $A\in \cW_+$. 
	Now the statement of the lemma is true for every for $F=f\otimes S_\ell(A)$, where $S_\ell(A)$ 
	is a limiting sequence of $A$, by Lemma \ref{lema5}. Hence, it holds for $f\otimes A$.
\end{proof}
	%With the same argument we can further assume that 
	%$A$ is positive operator. For $x\in \cD_0$ and a limiting sequence $S_\ell(A)$ of $A$, it holds that
		%\begin{eqnarray*}
			%&&M(K)(\id_{\cH})\overline{\psi}(f\otimes A)x
			%= \lim_{\ell\to\infty} M(K)(\id_{\cH})\overline{\psi}(f\otimes S_\ell(A))x\\
			%&=& \lim_{\ell\to\infty} M(K)(\id_{\cH})\sum_{k=1}^{n_\ell}
				 %\zeta_{k,\ell}\overline{\psi}(f\otimes P_{k,\ell})x\\
			%&=& \lim_{\ell\to\infty} \sum_{k=1}^{n_\ell}
				 %\zeta_{k,\ell}\overline{\psi}(f\otimes P_{k,\ell})M(K)(\id_{\cH})x\\
			%&=&\lim_{\ell\to\infty} \overline{\psi}(f\otimes S_\ell(A))M(K)(\id_{\cH})x
			%=\overline{\psi}(f\otimes A)M(K)(\id_{\cH})x,
		%\end{eqnarray*}
	%where the first and the fifth equality follow by the construction of the map $\overline{\psi}$,
	%the second and the fourth by the linearity of the map $\overline{\psi}$ and the third equality
	%follows by Lemma \ref{lema5}. This proves the inclusion part of the lemma.
	%Since by Lemma \ref{lema5}, 
	%\begin{eqnarray*}
		%\lim_{\ell\to\infty} \sum_{k=1}^{n_\ell}
				 %\zeta_{k,\ell}\overline{\psi}(f\otimes P_{k,\ell})M(K)(\id_{\cH})x&=&
				%\lim_{\ell\to\infty} \sum_{k=1}^{n_\ell}
				 %\zeta_{k,\ell}\overline{\psi}(f\chi_{K}\otimes P_{k,\ell})x\\
		%&=&\overline{\psi}(f\chi_{K}\otimes A)x,
	%\end{eqnarray*}
	%the equality part of the lemma also holds.
%\end{proof}

\begin{lemma}\label{lema6}
	For $F, G\in \cI(M)\otimes \cW$ we have 
		$$\overline{\psi}(FG)=\overline{\psi}(F)\overline{\psi}(G).$$
\end{lemma}

\begin{proof}
	By the linearity of $\overline{\psi}$, we may assume $F=f\otimes A$, $G=g\otimes B$, where
	$f,g\in\cI(M)$ and $A,B\in \cW$. Let $y\in \cD_0$. Then there is a compact set $K$ 
	and and $x\in \cK$, such that $y=M(K)(\id_{\cH})x$. 
	%such that  By the $\sigma$-compactness of $X$
	%there is a sequence $K_n$ of compact sets, such that $K_1=K$, $K_n\subseteq K_{n+1}$ for every 
	%$n\in\NN$ and $X=\cup_{n\in\NN}K_n$. 
	Let $\Delta_1:=K$ and $\Delta_m$ be a common bounding sequence for $fg$, $f$ and $g$.
	For a fixed $n\in \NN$,
	we have
		\begin{eqnarray*}
			  &&\overline{\psi}(fg\otimes AB)y=\\
				&=&\lim_{m\to\infty}\overline{\psi}(fg \chi_{\Delta_m}\otimes AB)y\\
				&=&\lim_{m\to\infty}\overline{\psi}(f \chi_{\Delta_m}\otimes A)\circ
					\overline{\psi}(g \chi_{\Delta_m}\otimes B)y
				\\
				&=&\lim_{m\to\infty}\overline{\psi}(f \chi_{\Delta_m}\otimes A)
					\circ\overline{\psi}(g \chi_{\Delta_m}\otimes B)
					\circ M(\Delta_n)(\id_{\cH})\circ M(K)(\id_{\cH})x\\
				&=&\lim_{m\to\infty}\overline{\psi}(f \chi_{\Delta_n}\otimes A)
					\circ \overline{\psi}(g \chi_{\Delta_m}\chi_{\Delta_n}\otimes A)
					\circ M(K)(\id_{\cH})x\\
				&=&\overline{\psi}(f \otimes A)
					\circ \overline{\psi}(g \chi_{\Delta_n}\otimes B)
					\circ M(K)(\id_{\cH})x
		\end{eqnarray*}
	where we used Lemma \ref{lema2} in the first and the fifth equality,
	\cite[Proposition 7.2.]{Zalar2014} in the second equality together with the fact that
	$\im(\overline{\psi})\subset \cD_0$ 
	(Indeed,
	\begin{eqnarray*}
		M(K)(\id_{\cH})\overline{\psi}(f\otimes A)M(K)(\id_{\cH})x&=&
			\overline{\psi}(f\otimes A)M(K)(\id_{\cH})^2x\\
			&=&\overline{\psi}(f\otimes A)M(K)(\id_{\cH})x.),
	\end{eqnarray*}
	the fact that 
	$M$ is a non-negative spectral measure in the third inequality (and $\Delta_n\supseteq \Delta_1$) and
	the equality part of Lemma \ref{lema2} in the fourth equality. As $n\to \infty$, we get, by the use of 
	Lemma \ref{lema2}, $\overline{\psi}(fg\otimes AB)y=\overline{\psi}(f \otimes A)
					\overline{\psi}(g \otimes B)y.$
	This concludes the proof.
\end{proof}

\begin{proof} [Proof of Theorem \ref{izrek4}]
	For $x\in \cD_0$ and $F:=\sum_{i=1}^n f_i\otimes A_i$, where $f_1,\ldots,f_n\in\cI(M)$, 
	$A_1,\ldots, A_n\in  \cW$, it is true that
		\begin{eqnarray*}
			\mathbb{I}_{M}(F^\ast)x
				&=&\overline{\psi}(F^{\ast})x
				 =\sum_{i=1}^n \overline{\psi}(\overline{f_i}\otimes A_i^{\ast})x
				 =\sum_{i=1}^n \overline{\psi}(f_i\otimes A_i)^\ast x\\
				&=&(\sum_{i=1}^n \overline{\psi}(f_i\otimes A_i))^\ast x
				 =\overline{\psi}(F)^\ast x
				 =\mathbb{I}_{M}(F)^{\ast}x,
		\end{eqnarray*}
	where	the first and the sixth equality follow by the definition of $\overline\psi$ and $\mathbb{I}_{M}$,
	the second and the fifth by the linearity of $\overline\psi$, the third equality follows by 
	Lemma \ref{lema4} and in the forth equality we used
	\cite[Proposition 1.6(vi)]{Sch2012}
	($\cD_0=\cD(\sum_{i=1}^n \overline{\psi}(f_i\otimes A_i))$
	is dense in $\cK$.).
	Since $\mathbb{I}_{M}(F)^{\ast}$ is the closed extension of the operator $\overline{\psi}(F^{\ast})$,
	and $\mathbb{I}_{M}(F^\ast)$ is its closure, part (1) is true. 
	
	Now we prove $\mathbb{I}_{M}(F+G)=\overline{\mathbb{I}_{M}(F)+\mathbb{I}_{M}(G)}$.
	%Using part (i), we have
		%$$\cD(\left(\mathbb{I}_{M}(F)+\mathbb{I}_{M}(G))^\ast)\supseteq
			%\cD(\mathbb{I}_{M}(F)^\ast+\mathbb{I}_{M}(G)^\ast)\supseteq
			%\cD(\mathbb{I}_{M}(F^\ast)+\mathbb{I}_{M}(G^\ast))\supseteq \cD_0.$$
	%Since $\cD_0$ is dense, $\mathbb{I}_{M}(F)+\mathbb{I}_{M}(G)$ is closable (by 
	%\cite[Proposition 1.8(i)]{Sch2012}).
	We have 
	\begin{eqnarray*}
		&&\mathbb{I}_{M}(F+G)=\overline{\overline{\psi}(F+G)}=\overline{\psi}(F+G)^{\ast\ast}	
		\supseteq
		(\overline{\psi}(F)^{\ast}+\overline{\psi}(G)^{\ast})^\ast\\
		&\supseteq&
		\overline{\psi}(F)^{\ast\ast}+\overline{\psi}(G)^{\ast\ast}=
		\overline{\overline{\psi}(F)}+ \overline{\overline{\psi}(G)}=
		\mathbb{I}_{M}(F)+\mathbb{I}_{M}(G)
		\supseteq 
		\overline{\psi}(F)+\overline{\psi}(G),
	\end{eqnarray*} 
	where the first $\supseteq$ follows by $\cD_0\subseteq \cD(\overline{\psi}(F)^{\ast})$
	for every $F\in \cB\otimes \cW$, by Lemma \ref{lema4}.
	Hence, $\mathbb{I}_{M}(F)+\mathbb{I}_{M}(G)$ is closable and the equality 
	$\mathbb{I}_{M}(F+G)=\overline{\mathbb{I}_{M}(F)+\mathbb{I}_{M}(G)}$ is true.
	
	%
	%Since $\mathbb{I}_{M}(F+G)$ is the closure of $\mathbb{I}_{M}(F+G)|_{\cD_0}$ and
	%$\overline{\mathbb{I}_{M}(F)+\mathbb{I}_{M}(G)}$ is the closed extension of $\mathbb{I}_{M}(F+G)|_{\cD_0}$, part (ii) is true.
	%For the inclusion $\mathbb{I}_{M}(F+G)\supseteq 
	%\overline{\mathbb{I}_{M}(F)+\mathbb{I}_{M}(G)}$ notice that
	%$\mathbb{I}_{M}(F+G)=\overline{\overline{\psi}(F+G)}=\overline{\psi}(F+G)^{\ast\ast}\supseteq
	%(\overline{\psi}(F)^{\ast}+\overline{\psi}(G)^{\ast})^\ast\supseteq
	%\overline{\psi}(F)^{\ast\ast}+\overline{\psi}(G)^{\ast\ast}=
	%\overline{\overline{\psi}(F)}+ \overline{\overline{\psi}(G)}=
	%\overline{\mathbb{I}_{M}(F)+\mathbb{I}_{M}(G)}.$
	%Since 
	
	By Lemma \ref{lema6}, $\cD_0$ is contained in the domain of the operator $\mathbb{I}_{M}(G^\ast)\mathbb{I}_{M}(F^\ast)$. 
	Since $(\mathbb{I}_{M}(F)\mathbb{I}_{M}(G))^\ast\supseteq \mathbb{I}_{M}(G)^\ast \mathbb{I}_{M}(F)^\ast$ (by 
	\cite[Proposition 1.7(i)]{Sch2012}) and also
		$\mathbb{I}_{M}(G)^\ast \mathbb{I}_{M}(F)^\ast \supseteq 
	\mathbb{I}_{M}(G^\ast) \mathbb{I}_{M}(F^\ast)$ by part (1), the operator 
	$\mathbb{I}_{M}(F)\mathbb{I}_{M}(G)$ is closable (by \cite[Proposition 1.8(i)]{Sch2012}).
	For $x\in \cD_0$, $\mathbb{I}_{M}(FG)x=\mathbb{I}_{M}(F)\mathbb{I}_{M}(G)x$ (by Lemma \ref{lema6}).
	Since $\mathbb{I}_{M}(FG)$ is the closure for $\overline{\psi}(FG)$ and $\overline{\mathbb{I}_{M}(F)\mathbb{I}_{M}(G)}$
	is the closed extension for $\overline{\psi}(FG)$, part (3) follows.
\end{proof}

\section{Proof of Theorem \ref{theorem-D}}
\label{resitev-problema-D}

Assume the notation as in Theorem \ref{theorem-D}. In the proof we will need the following observation.

\begin{proposition} \label{omejenost-funkcionalov}
		For every $k\in \cK$, such that for all $b\in\cB$ the maps
			$$(\overline{\rho_{b}})_k: \cW\to \CC,\quad 
				(\overline{\rho_{b}})_k(A):=\left\langle \overline{\rho(b\otimes A)}\;k,k\right\rangle$$ 
	are well-defined, the map $(\overline{\rho_{b}})_k$ is a bounded linear functional.	
\end{proposition}

\begin{proof}
		Choose $k\in \cK$, such that for all $b\in \cB$ the maps 
	$(\overline{\rho_{b}})_k$ are well-defined.
	Fix $b\in \cB$. First we argue, that $(\overline{\rho_{b}})_k$ is a linear functional. The homogenuity of $(\overline{\rho_{b}})_k$ is clear. For additivity we notice that
		\begin{eqnarray*}
			&&\overline{\rho(b\otimes (A+B))}= \overline{\rho(b\otimes A)+\rho(b\otimes B)}=(\rho(b\otimes A)+\rho(b\otimes B))^{\ast\ast}\\
				%&=&	(\left(\rho(b\otimes A)+\rho(b\otimes B))^{\ast})^\ast
			&\supseteq& (\rho(b\otimes A)^\ast+\rho(b\otimes B)^\ast)^\ast
			\supseteq	\rho(b\otimes A)^{\ast\ast}+\rho(b\otimes B)^{\ast\ast}=\overline{\rho(b\otimes A)}+\overline{\rho(b\otimes B)},
		\end{eqnarray*}	
	where we used \cite[Theorem 1.8(ii)]{Sch2012} for the second and the third equality and \cite[Proposition 1.6(vi)]{Sch2012} for both inclusions $\supseteq$.

  The proof of boundedness of $(\overline{\rho_{b}})_k$ is the same to the proof of boundedness
of $(\rho_{b})_k$ in Proposition \ref{trditev-o-omejenosti} above.
	%Now we prove $(\overline{\rho_{b}})_k$ is bounded. Write $b=\frac{(b+1)^\ast(b+1)}{2}-\frac{(b-1)^\ast(b-1)}{2}$. Denote
	%$c=\frac{(b+1)^\ast(b+1)}{2}$ and $d=\frac{(b-1)^\ast(b-1)}{2}$. Fix $A\in \cW$. We have
	%\begin{eqnarray*}
			%&&\overline{\rho((c-d)\otimes A)}= \overline{\rho(c\otimes A)-\rho(d\otimes A)}=(\rho(c\otimes A)-\rho(d\otimes A))^{\ast\ast}\\
			%%=	(\left(\rho(c\otimes A)-\rho(d\otimes A))^{\ast})^\ast
			%&\supseteq& (\rho(c\otimes A)^\ast-\rho(d\otimes A)^\ast)^\ast
			%\supseteq	\rho(c\otimes A)^{\ast\ast}-\rho(d\otimes A)^{\ast\ast}=\overline{\rho(c\otimes A)}-\overline{\rho(d\otimes A)},
		%\end{eqnarray*}	
	%by the same reasoning as above. Hence, $(\overline{\rho_{b}})_k=(\overline{\rho_{c}})_k-(\overline{\rho_{d}})_k$.
	%Since $(\overline{\rho_{c}})_k,(\overline{\rho_{d}})_k: \cW\to \CC$ are 
	%positive linear functionals,
	%they are bounded, by \cite[5.12. Corollary]{Con}. Hence, $(\overline{\rho_{b}})_k$ is bounded.
\end{proof}

We also need the following easy observation.

\begin{proposition} \label{dobra-definiranost-P-jev-na-k-jih}
		For every $b\in \cB$ we have 
	$\cD(\overline{\rho(b\otimes \id_{\cH})})\subseteq 
	 \cD(\overline{\rho(b\otimes A)})$ 
	for every $A\in \cW$.
		In particular, $\cup_{K\in \scrK}\cD_{\alpha_K, \rho_{\id_{\cH}}}\subseteq 
	\cD(\overline{\rho(b\otimes A)})$ for every $b\in \cB$ and every
	$A\in \cW$.
\end{proposition}

\begin{proof}
	First let us prove, that for every $b\in \cB$ and every hermitian projection $P\in \cW_p$
	the inclusion 
	$\cD(\overline{\rho(b\otimes \id_{\cH})})\subseteq 
	 \cD(\overline{\rho(b\otimes P)})$ of domains is true.	
	By the linearity of $\rho$,
		$\rho(b\otimes \id_{\cH})=\rho(b\otimes P)+\rho(b\otimes 
			(\id_\cH-P)).$ 
	By $\rho$ being a $\ast$-representation on $\cD$, for every $k_1,k_2\in \cD$ the following is true
		\begin{eqnarray*}
			&&\left\langle\right. \rho(b\otimes P)k_1,
			\rho(b\otimes(\id_\cH-P))k_2\left.\right\rangle =
			\left\langle \right.
			\rho(b\otimes(\id_\cH-P))^\ast\rho(b\otimes P)k_1,
			k_2\left.\right\rangle\\
			&=&	\left\langle\right. \rho(b^\ast b\otimes(\id_\cH-P)P)k_1,k_2 \left.\right\rangle
			=	\left\langle \right.\rho(0)k_1,k_2 \left.\right\rangle=0.
		\end{eqnarray*}
	Hence, 
		$\rho(b\otimes P)\cD \perp \rho(b\otimes (\id_\cH-P))\cD.$
	Thus, the inclusion $\cD(\overline{\rho(b\otimes \id_{\cH})})\subseteq 
	 \cD(\overline{\rho(b\otimes P)})$ holds.
	
	Now we will prove the inclusion 
	$\cD(\overline{\rho(b\otimes \id_{\cH})})\subseteq 
	 \cD(\overline{\rho(b\otimes A)})$ of domains for an arbitrary operator $A\in \cW$.
	Let $k\in \cD(\overline{\rho(b\otimes \id_{\cH})})$. Let $k_i$ be the sequence from $\cD$,
	such that $k_i\to k$ and $\rho(b\otimes \id_{\cH})k_i\to \rho(b\otimes \id_{\cH})k$. For the inclusion
	$k\in \cD(\overline{\rho(b\otimes A)})$, the sequence
	$\rho(b\otimes A)k_i$ has to be convergent. We have
		$$\|\rho(b\otimes A)k_i-\rho(b\otimes A)k_j\|=
			\left\langle\right. \rho(b^\ast b\otimes A^\ast A)(k_i-k_j),k_i-k_j\left.\right\rangle$$
	Let $S_\ell(A^\ast A)=\sum_{k=1}^{n_\ell} \zeta_{k,\ell} R_{k,\ell}$ 
	be a limiting sequence of $A^\ast A$. It can be chosen to be increasing.
	Then also the sequence 
		$\left\langle \rho(b^\ast b\otimes S_\ell(A^\ast A))(k_i-k_j),k_i-k_j\right\rangle$
	is increasing and converges to $\left\langle \rho(b^\ast b\otimes A^\ast A)(k_i-k_j),k_i-k_j\right\rangle$, by Proposition \ref{omejenost-funkcionalov}.
	For every $\epsilon>0$ we can choose $N\in \NN$, such that for $i,j>N$
		$\left\langle \right.\rho(b^\ast b\otimes \id_{\cH})(k_i-k_j),k_i-k_j\left.\right\rangle<\frac{\epsilon}{\|
			A^\ast A\|}.$
	Hence, for every $\ell\in \NN$ we have
	\begin{eqnarray*}
		\left\langle \rho(b^\ast b\otimes S_\ell(A^\ast A))(k_i-k_j),k_i-k_j
			\right\rangle
		&=& \sum_{k=1}^{n_\ell}\zeta_{k,\ell}
			\left\langle \rho(b^\ast b\otimes R_{k,\ell})(k_i-k_j),k_i-k_j\right\rangle\\
		&\leq& \|A^\ast A\|
			\left\langle \rho(b^\ast b\otimes \id_{\cH})(k_i-k_j),k_i-k_j\right\rangle\\
		&\leq& \epsilon.
	\end{eqnarray*}
	Hence, $\rho(b\otimes A)k_i$ is convergent and so $k\in \cD(\overline{\rho(b\otimes A)})$.
\end{proof}

\begin{proof}[Proof of Theorem \ref{theorem-D}]
	%Direction $(2)\Rightarrow (1).$ Since $M$ has a compact support, $f_F$ is $M$-integrable for every 
%$F\in \cB\otimes B(\cH)$. For $\rho$ defined by (2), we have to prove the linearity,
%the multiplicativity and the equality
	%$\rho(F^\ast)=\rho(F)^{\ast}$ for every $F\in \cB\otimes W_1$. The linearity follows by 
%\cite[Proposition 3.5]{Zalar2014}, while the multiplicativity by 
%\cite[Proposition 7.2]{Zalar2014}. To show 
%$\int_{\cB} f_{F}^\ast\; dM=(\int_{\cB} 	f_F\;dM)^{\ast}$  
	%it suffices, by the linearity, to consider elements of the form $F=b\otimes A$, $A\in B(\cH)_+$. Since $M_A$ is a positive operator-valued measure, we have
	%$\int_\cB (f_b\otimes A)^\ast\; dM=\int_\cB (f_b\otimes A)\;dM=(\int_\cB (f_b\otimes A)\;dM)^{\ast}$ and the result follows.
	Direction $(2)\Rightarrow (1)$ follows by Propositions \ref{integrabilnost} and \ref{gostost-D-ja}.
	
	Direction $(1)\Rightarrow (2).$ By the assumption, the maps
		$\rho_P$ are integrable $\ast$-rep\-re\-sen\-ta\-tions 
	for every $P\in \cW_p$.	
	By Theorem \ref{theorem-C'}, there exist unique regular spectral measures
	$E_P:\B(\widehat{B})\to B(\cK)$
	such that $\overline{\rho_P(b)}x=\int_{\widehat{B}} f_b(\chi)\; dE_P(\chi)\;x$ holds for every 
	$x\in \cD(\overline{\rho_P(b)})$, $b\in \cB$,
	$P\in \cW_p$ and the support $\supp((E_{\id_{\cH}})_{x,x})$ is compact for 
	every $k\in \cup_{K\in \scrK}\cD_{\alpha_K, \rho_{\id_{\cH}}}$. 
	The idea is to show that the family
	$\{E_P\}_{P\in \cW_p}$ satisfies the conditions of Theorem \ref{karakterizacija-nenegativnih-spektralnih-mer} to obtain a non-negative spectral measure $M$ representing $\rho$.
	
	By an analoguous proof as in the proof
	of Theorem \ref{theorem-C'}, we prove the containment 
	$\supp((E_P)_{k,k})\subseteq \supp((E_{\id_{\cH}})_{k,k})$ for every $P
	\in \cW_p$ and every 
	$k\in\cup_{K\in \scrK}\cD_{\alpha_K, \rho_{\id_{\cH}}}$.

	It remains to check first that the family $\{E_P\}_{P\in W_p}$ satisfies the conditions
of Theorem \ref{karakterizacija-nenegativnih-spektralnih-mer}, second that $M$ 
is a representing measure of $\overline{\rho}$ and finally that $M$
is unique, regular and normalized. Since the support of every $E_P$ is contained in
the compact set $K_k = \supp(E_{\id_\cH})_{k,k}$ for every $k\in \cup_{K\in \scrK}\cD_{\alpha_K, \rho_{\id_{\cH}}}$, the proofs of the conditions of Theorem \ref{karakterizacija-nenegativnih-spektralnih-mer} are the same as the proofs of the corresponding conditions of \cite[Theorem 9.1]{Zalar2014},
just that we replace the use of \cite[Lemma 2.3]{Zalar2014} by
Lemma \ref{gostost-krogle} above. By the polarization argument and by the density of
$\cup_{K\in \scrK}\cD_{\alpha_K, \rho_{\id_{\cH}}}$, they are true for every $k\in \cK$.

	\textbf{$M$ is the representing measure of $\rho$:}
	We have 
		\begin{eqnarray*}
			\overline{\rho(\sum_{i=1}^n b_i\otimes A_i)} &=&
									\rho(\sum_{i=1}^n b_i\otimes A_i)^{\ast\ast}=
									(\sum_{i=1}^n \rho(b_i\otimes A_i))^{\ast\ast}\\
			&\supseteq&	\sum_{i=1}^n \rho(b_i\otimes A_i)^{\ast\ast}=
									\sum_{i=1}^n \overline{\rho(b_i\otimes A_i)}.			
		\end{eqnarray*}
	Hence,
		$\overline{\rho(\sum_{i=1}^n b_i\otimes A_i)}=
			\overline{\sum_{i=1}^n \overline{\rho(b_i\otimes A_i)}}.$
	Therefore, it suffices to prove that
	$\overline{\rho(b\otimes A)}=\int_X (f_b\otimes A)\; dM$.
	For a hermitian projection $A\in \cW_p$, this is true by the construction of measures $M_A$.
	Let now $A\in \cW_+$ be a positive operator. By Proposition \ref{omejenost-funkcionalov},
	$(\overline{\rho_b})_k$  is bounded for every 
	$k\in \cup_{K\in \scrK}\cD_{\alpha_K, \rho_{\id_{\cH}}}.$
	Let $K_0=\supp((E_{\id_{\cH}})_{k,k})$.
	Therefore,
		\begin{eqnarray*}
			&	&	(\overline{\rho_b})_k(A)
			 = \left\langle\right. \overline{\rho(b\otimes A)}k,k \left.\right\rangle
			 =	\lim_\ell \left\langle\right. \overline{\rho(b\otimes S_\ell(A))}k,k \left.\right\rangle\\
			&=&	\lim_\ell \left\langle\right. (\int_{X} (f_b\otimes S_\ell(A)) \; dM)\; 
					k, k\left.\right\rangle
			 =	\lim_\ell \left\langle\right. (\int_{K_0} 
					(f_b\otimes S_\ell(A)) \; dM)\; k, k\left.\right\rangle\\
			&=&	\left\langle\right. (\int_{K_0} 
					(f_b\otimes A) \; dM)\; k, k\left.\right\rangle
			 = \left\langle\right. (\int_{X} 
					(f_b\otimes A) \; dM)\; k, k\left.\right\rangle
		\end{eqnarray*}
	By the polarization argument, 
		$\left\langle \overline{\rho(b\otimes A)}k_1,k_2 \right\rangle=
		\left\langle (\int_{X} (f_b\otimes A) \; dM)\; k_1, k_2\right\rangle$
	for every $k_1, k_2 \in \cup_{K\in \scrK}\cD_{\alpha_K, \rho_{\id_{\cH}}}.$
	By the density of $\cup_{K\in \scrK}\cD_{\alpha_K, \rho_{\id_{\cH}}}$,
		$\overline{\rho(b\otimes A)}=\int_{X} (f_b\otimes A) \; dM$.
	We decompose an arbitrary $A$ into the usual linear combination of four positive parts and use the result for each of them.
	
	Finally, $M$ is unique, regular and normalized, by the uniqueness and the regularity 	
	of each $E_P$ and the unitality of $\rho$.
\end{proof}

\end{document}